\RequirePackage{ifthen}
\newboolean{SIOPT}
\setboolean{SIOPT}{false}

\ifthenelse {\boolean{SIOPT}}
{
\documentclass[review,hidelinks,onefignum,onetabnum]{siamart220329}
} {
\documentclass[11pt,letterpaper]{article}
\usepackage[margin=1in]{geometry}
}

\usepackage{graphicx}
\usepackage{amsmath}
\usepackage{amssymb}
\usepackage{hyperref}
\usepackage{color}

\usepackage{algorithm}
\usepackage{algpseudocode}
\usepackage{accents}

\usepackage{enumitem} % to use [noitemsep]
\setlist[itemize]{noitemsep,nolistsep} %,leftmargin=0mm}
\setlist[enumerate]{noitemsep,nolistsep} %,leftmargin=0mm}

\usepackage{cite} % to allow for spaces in citations

\usepackage{mathrsfs} % for mathscr

\usepackage{xspace} % for spaces after symbols

\usepackage{soul} % for \ul

\usepackage[makeroom]{cancel} % for \cancel

\usepackage{thmtools}
\usepackage{thm-restate}

\newcommand{\R}{\mathbb R}
\newcommand{\Z}{\mathbb Z}
\newcommand{\Q}{\mathbb Q}

\let\st\relax
\DeclareMathOperator{\st}{s.t.}
\DeclareMathOperator{\conv}{conv}
\DeclareMathOperator{\proj}{proj}
\DeclareMathOperator{\spn}{span}
\DeclareMathOperator{\size}{size}
\DeclareMathOperator{\width}{width}

\DeclareMathOperator{\dist}{dist}
\DeclareMathOperator{\vol}{vol}
\DeclareMathOperator{\rec}{rec.cone}
\DeclareMathOperator{\rank}{rank}

\newcommand{\A}{\mathcal A}
\newcommand{\F}{\mathcal F}
\newcommand{\G}{\mathcal G}
\renewcommand{\P}{\mathcal P}
\newcommand{\B}{\mathcal B}
\newcommand{\C}{\mathcal C}

\renewcommand{\S}{\mathcal S}

\renewcommand{\H}{\mathcal H}

\newcommand{\sQ}{\mathcal Q}

\newcommand{\basmatB}{B}

\newcommand{\ceil}[1]{\lceil#1\rceil}

\newcommand{\abs}[1]{\lvert#1\rvert}
\newcommand{\absL}[1]{\left\lvert#1\right\rvert}

\newcommand{\norm}[1]{\lVert#1\rVert_2}

\newcommand{\infnorm}[1]{\lVert#1\rVert_\infty}

\newcommand{\pare}[1]{\left(#1\right)}
\newcommand{\bra}[1]{\left\{#1\right\}}

\newcommand{\transp}{\mathsf T}

\newcommand{\constlen}{2^{p(p-1)/4}}
\ifthenelse {\boolean{SIOPT}}
{
\newenvironment{prfc}[1][]
{\begin{proof}[Proof #1]}
{\end{proof}}

\newenvironment{cpf}
{\begin{trivlist} \item[] {\em Proof of claim.}}
{$\hfill\diamond$ \end{trivlist}}

} {
\usepackage{amsthm}
\newtheorem{theorem}{Theorem}
\newtheorem{proposition}{Proposition}
\newtheorem{lemma}{Lemma}

\newenvironment{prfc}[1][]
{\begin{proof}[Proof #1]}
{\end{proof}}

\newenvironment{cpf}
{\begin{trivlist} \item[] {\em Proof of claim. }}
{$\hfill\diamond$ \end{trivlist}}

\usepackage[capitalize,noabbrev]{cleveref}
\crefname{observation}{Observation}{Observations}
\crefname{claim}{Claim}{Claims}

}

\crefname{problem}{Problem}{Problems}

\newtheorem{claim}{Claim}

\title{Convex quadratic sets and the complexity of mixed integer convex quadratic programming\thanks{Supported by AFOSR grant FA9550-23-1-0433. Any opinions, findings, and conclusions or
recommendations expressed in this material are those of the authors and do not necessarily reflect the views of the Air Force Office of Scientific Research.}}

\author{Alberto Del Pia
\thanks{Department of Industrial and Systems Engineering \& Wisconsin Institute for Discovery,
             University of Wisconsin-Madison, Madison, WI, USA.
             E-mail: {\tt delpia@wisc.edu}.}}
\date{February 2, 2024}

\begin{document}

%\tableofcontents

%\begin{titlepage}
\maketitle

\begin{abstract}
In pure integer linear programming it is often desirable to work with polyhedra that are full-dimensional, and it is well known that it is possible to reduce any polyhedron to a full-dimensional one in polynomial time.
More precisely, using the Hermite normal form, it is possible to map a non full-dimensional polyhedron to a full-dimensional isomorphic one in a lower-dimensional space, while preserving integer vectors.
In this paper, we extend the above result simultaneously in two directions.
First, we consider mixed integer vectors instead of integer vectors, by leveraging on the concept of ``integer reflexive generalized inverse.''
Second, we replace polyhedra with convex quadratic sets, which are sets obtained from polyhedra by enforcing one additional convex quadratic inequality.
We study structural properties of convex quadratic sets, and utilize them to obtain polynomial time algorithms to recognize full-dimensional convex quadratic sets, and to find an affine function that maps a non full-dimensional convex quadratic set to a full-dimensional isomorphic one in a lower-dimensional space, while preserving mixed integer vectors.
We showcase the applicability and the potential impact of these results by showing that they can be used to prove that mixed integer convex quadratic programming is fixed parameter tractable with parameter the number of integer variables.
Our algorithm unifies and extends the known polynomial time solvability of pure integer convex quadratic programming in fixed dimension and of convex quadratic programming.
\end{abstract}

%\thispagestyle{empty}

%\end{titlepage}

%\tableofcontents

\newcommand{\kywrds}{convex quadratic sets, mixed integer quadratic programming, FPT algorithm}

\ifthenelse {\boolean{SIOPT}}
{
% For SIOPT begin

% For SIOPT end
}{
% For OO begin
}

\section{Introduction}
\label{sec intro}

Polyhedra constitute one of the most fundamental geometric objects in mathematical optimization, especially in linear programming, integer programming, and polyhedral combinatorics.
A desirable property of a polyhedron is being full-dimensional or, in other words, having positive volume.
It is well known that this property can be checked in polynomial time. 
Furthermore, it is possible to find, in polynomial time, an affine function that maps a non full-dimensional polyhedron to a full-dimensional isomorphic one in a lower-dimensional space.
A map of this type can be found rather easily, for example by solving linear programming problems \cite{SchBookIP,ConCorZamBook}.
This technique has been used in many algorithms in the literature to assume, without loss of generality, that a given polyhedron is full-dimensional.

A similar reduction is known also for pure integer linear sets.
In fact, it is possible to find, in polynomial time, an affine function as the one discussed above, but with the additional property of preserving integer vectors.
This result is heavily used in pure integer programming and, for example, plays a key role in most Lenstra-type algorithms (see, e.g., \cite{Len83}).
A map of this type can be obtained using a characterization of integer solutions to systems of linear equations, through the use of the Hermite normal form (see, e.g., \cite{SchBookIP,ConCorZamBook}).
%This characterization is then exploited to construct
%an affine function that maps the integer solutions to a system of linear equations to all the integer points in a lower-dimensional space.

To the best of our knowledge, no similar result is known in the mixed integer setting.
% no extension of these results to the mixed integer setting are currently unknown.
In \cref{sec linear} we close this gap in the literature and extend the above results to mixed integer linear sets, by leveraging on the concept of ``integer reflexive generalized inverse.''
In particular, in \cref{prop gen solution} we give a polynomial time algorithm which constructs an affine function that maps the mixed integer solutions to a system of linear equations to all the mixed integer points in a lower-dimensional space.
%provides a polynomial time characterization of \emph{mixed integer} solutions to systems of linear equations.
This result is then used to prove \cref{th full dim poly}, which allows us to find in polynomial time an affine function that maps a non full-dimensional polyhedron to a full-dimensional isomorphic one in a lower-dimensional space, while preserving mixed integer vectors.
%as the one discussed above, but that preserves \emph{mixed integer} vectors.
The latter result is very versatile, and can be used in the design of polynomial time algorithms for mixed integer programming problems with linear inequality constraints to assume, without loss of generality, that the given polyhedron is full-dimensional.

A natural question is whether results similar to \cref{th full dim poly} can be obtained for convex sets that are more general than polyhedra.
In particular, results of this type would yield reductions similar to the previous ones, but in the wider realm of mixed integer programming problems with convex inequality constraints. 
In turn, as we will showcase later, this would allow us to make use of a number of techniques that have been so far applied only to systems of linear constraints.
%While in \cref{sec linear} we only consider polyhedra, 
In \cref{sec cqs} we turn our attention to convex quadratic sets, which are sets obtained from polyhedra by enforcing one additional convex quadratic inequality.
%In algorithms, convex quadratic sets in input or output are always given through their defining integer data as above.
%Similarly, polytopes and hyperplanes are always given via an external description with integer data. \note{need?}
%A natural next step is to consider inequalities that are quadratic instead of linear.
On the one hand, convex quadratic sets form one of the simplest and most natural extensions of polyhedra.
% is obtained by replacing one linear inequality constraint with a convex quadratic inequality.
On the other hand, these sets play a fundamental role in (mixed integer) 
quadratic programming (see, e.g., \cite{Vav90,dPDeyMol17MPA}),
%as we will discuss later in the paper.
%a connection that we will explore later in the paper.
% \cite{dpmany}, 
and are linked with several other branches of optimization, including (mixed integer) second-order cone programming (see, e.g., \cite{AliGol03,HoKil17}), and cutting plane methods (see, e.g., \cite{MunSer22}).
%
%Furthermore, quadratic functions are at the 
%
In \cref{lem low dim 1,lem low dim 2,lem full-dim} we categorize convex quadratic sets into three types and we obtain structural and algorithmic results for each type.
These three lemmas allow us to give,
%To obtain an extension of \cref{th full dim poly} to convex quadratic sets, we first need to be able to recognize when a convex quadratic sets is full-dimensional.
%While it is well known how to check in polynomial time whether a polyhedron is full-dimensional, this task is significantly more involved for convex quadratic sets.
in \cref{prop full-dim}, a characterization of full-dimensional convex quadratic sets that yields a polynomial time algorithm to check full-dimensionality.
These results are then used to 
%Our characterization of full-dimensional convex quadratic sets plays 
%a key role in 
prove our extension of \cref{th full dim poly} to convex quadratic sets.
Namely, in \cref{th full-dim conv quad}, we show how we can find in polynomial time an affine function that maps a non full-dimensional convex quadratic set to a full-dimensional isomorphic one in a lower-dimensional space, while preserving mixed integer vectors.
%\note{sell techniques?}

The fundamental results obtained in \cref{sec linear,sec cqs} are not only mathematically interesting in their own right, but also enrich %enhance 
the mathematical toolkit available to researchers for the design and analysis of algorithms for mixed integer linear and quadratic programming.
%MIQP problems. Moreover, this work will lay the foundations for the design of solution methods for the more general class of MINLP problems.
%We back up this claim by showing,
In \cref{sec proof of main}, we showcase the applicability and the potential impact of these results.
In particular, we explain how these results can be used to revive Lenstra's original approach for ellipsoid rounding presented in \cite{Len83}, in the context of mixed integer convex quadratic programming (MICQP), which is the problem of minimizing a convex quadratic function over the mixed integer points in a polyhedron.
In turn, this allows us to design, in \cref{th main opt}, an algorithm for MICQP that is fixed parameter tractable (FPT) with parameter the number of integer variables.
This approach has the key advantage that it does not use the ellipsoid method as a subroutine.
% we show that 
%they can be used to prove that 
%%Next, in \cref{th main opt}, we settle the computational complexity of 
%mixed integer convex quadratic programming (MICQP) is fixed parameter tractable (FPT) with parameter the number of integer variables.
%More precisely, in \cref{sec proof of main}, we present an algorithm that accurately solves MICQP, which 
%with little additional effort.
%The proof of \cref{th main opt} is given .
%In \cref{sec MICQP feas prob} 
%As a consequence, 
\cref{th main opt} implies that,
when the number of integer variables is fixed (but the number of continuous variables is not), our algorithm is polynomial time.
Our result unifies and extends the known polynomial time solvability of pure integer convex quadratic programming in fixed dimension \cite{Kha83} and of convex quadratic programming \cite{KozTarKha81}.
Furthermore, it is tight, in the sense that MICQP is NP-hard if the number of integer variables are not fixed \cite{SchBookIP}.
%An important feature of our algorithm is that it does not use any ``strong hammer'' besides Lenstra's flatness result (see \cref{lem Lenstra}); in particular, it does not use the ellipsoid method as a subroutine.
Our algorithm can also be used as a powerful subroutine for the design of algorithms with theoretical guarantees for more general mixed integer nonlinear programming problems. 
In particular, it is a necessary tool to generalize recent approximation algorithms for mixed integer nonconvex quadratic programming \cite{dP16IPCO,dP18MPB,dP19SIOPT,dP23bMPA} beyond the fixed rank setting.

Our proof of \cref{th main opt} 
%is presented in \cref{sec proof of main}.
% and utilizes on \cref{th full dim poly,th full-dim conv quad}.
consists of three main building blocks: \cref{prop sandwich}, \cref{prop main feasibility}, and \cref{prop unbounded}.
\cref{prop sandwich} is a result of independent interest: a L\"owner-John-type algorithm for projected convex quadratic sets.
%, which is presented in \cref{sec cqs LJ}. 
Given a convex quadratic set $\sQ$ in $\R^n$ that is bounded and full-dimensional, \cref{prop sandwich} gives an algorithm to construct two concentric ellipsoids in $\R^p$ that sandwich the projection of $\sQ$ onto $\R^p$, and whose ratio is bounded by $4\ceil{\sqrt{p}}^3$.
%, where $p$ denotes the number of integer variables.
In \cref{prop main feasibility}, we show that there is an FPT algorithm with parameter the number of integer variables that solves the feasibility problem associated with \cref{prob MICQP}, that is, the problem of determining whether a convex quadratic set contains mixed integer points or not.
%Leveraging on the above results, 
%\cref{prop full-dim,prop full-dim conv quad}, 
%Another key result at the heart of the feasibility algorithm is .
%Next, we get back to \cref{prob MICQP}.
In \cref{prop unbounded} we characterize when \cref{prob MICQP} is bounded and obtain an FPT algorithm to check it.
\section{Systems of linear equations and polyhedra}
\label{sec linear}

%In this section, we see why we can assume without loss of generality that the polyhedron $\bra{x \in \R^n : Wx \le w}$ in \cref{prob MIQP} is full-dimensional.

%\subsection{Reflexive generalized inverse}

The main objective of this section is to state and prove \cref{th full dim poly},
%, which was informally described in \cref{sec intro}.
which allows us to find in polynomial time an affine function that maps a non full-dimensional polyhedron to a full-dimensional isomorphic one in a lower-dimensional space, while preserving mixed integer vectors.
Our first goal is to obtain, in polynomial time, an affine function that maps 
%non full-dimensional polyhedron to a full-dimensional one in a lower-dimensional space,
%give a polynomial time algorithm that characterizes
the set of mixed integer solutions to a system of linear equations to all the mixed integer vectors in a lower-dimensional space.
To do so, we use the concept of reflexive generalized inverse.
%
%Let $A \in \R^{m \times n}$.
A \emph{reflexive generalized inverse} of a matrix $A \in \R^{m \times n}$ is a matrix denoted by $A^{\#} \in \R^{n \times m}$ such that $AA^{\#}A = A$ and $A^{\#}AA^{\#} = A^{\#}$.
%\begin{align*}
%AA^{\#}A = A \qquad \text{and} \qquad A^{\#}AA^{\#} = A^{\#}.
%\end{align*}
A reflexive generalized inverse of $A$ is said to be \emph{integer} if $A^{\#}A$ is integer, and it is denoted by $A^{\#}_I$.
We show that an integer reflexive generalized inverse can be obtained in polynomial time.

\begin{lemma}
\label{lem poly inv}
Let $A \in \Q^{m \times n}$.
There is a polynomial time algorithm that finds an integer reflexive generalized inverse $A^{\#}_I \in \Q^{n \times m}$ of $A$.
The algorithm also finds a unimodular matrix $U \in \Z^{n \times n}$ such that $A^{\#}_I A = U \begin{bmatrix}I_r & 0 \\ 0 & 0\end{bmatrix} U^{-1}$, where $r$ is the rank of $A$.
\end{lemma}

\begin{proof}
Let $r$ be the rank of $A$, which can be computed in polynomial time with Gaussian elimination.
By interchanging rows of $A$, we can move $r \le m$ linearly independent rows in the first $r$ rows.
If we denote by $W$ the unimodular matrix of the row interchanges, we have
\begin{align}
\label{eq TU1}
WA =
\begin{bmatrix}
A_1 \\
A_2
\end{bmatrix} \in \Q^{m \times n},
\end{align}
where $A_1 \in \Q^{r \times n}$ has full row rank and $A_2 \in \Q^{m-r \times n}$.
%For ease of notation, let 
%\begin{align*}
%B:=
%\begin{bmatrix}
%A_1 \\
%A_2
%\end{bmatrix} \in \R^{m \times n}.
%\end{align*}

It follows from \cite{KanBac79} that we can compute in polynomial time a unimodular matrix $U \in \Z^{n \times n}$ such that $A_1 U=\begin{bmatrix}K_1 & 0_{r \times n-r}\end{bmatrix}$, where $K_1 \in \Q^{r \times r}$ is invertible.
Since each row of $A_2$ is a linear combination of rows of $A_1$, we have
\begin{align}
\label{eq TU2}
\begin{bmatrix}
A_1 \\
A_2
\end{bmatrix} 
U
=\begin{bmatrix}K_1 & 0_{r \times n-r} \\ K_2 & 0_{m-r \times n-r}\end{bmatrix},
\end{align}
where $K_2 \in \Q^{m-r \times r}$.
Let 
\begin{align*}
K := \begin{bmatrix}K_1 & 0_{r \times n-r} \\ K_2 & 0_{m-r \times n-r}\end{bmatrix} \in \Q^{m \times n}.
\end{align*}

Next, we show that an integer reflexive generalized inverse of $K$ is 
\begin{align*}
K^{\#}_I := \begin{bmatrix}K_1^{-1} & 0_{r \times m-r} \\ 0_{n-r \times r} & 0_{n-r \times m-r}\end{bmatrix}.
\end{align*}
We have
\begin{align*}
& K_{I}^{\#} K 
= \begin{bmatrix}I_r & 0 \\ 0 & 0\end{bmatrix}
 \qquad \text{integer}, \\
& K K_{I}^{\#} K
= \begin{bmatrix}K_1 & 0 \\ K_2 & 0\end{bmatrix} \begin{bmatrix}I_r & 0 \\ 0 & 0\end{bmatrix}
= K, \\
& K_{I}^{\#} K K_{I}^{\#}
= \begin{bmatrix}I_r & 0 \\ 0 & 0\end{bmatrix} \begin{bmatrix}K_1 & 0_{r \times n-r} \\ K_2 & 0_{m-r \times n-r}\end{bmatrix}
=  K_{I}^{\#}.
\end{align*}

We now prove that an integer reflexive generalized inverse of $A$ is 
\begin{align*}
A^{\#}_I := U K^{\#}_I W.
\end{align*}
From \eqref{eq TU1} and \eqref{eq TU2}, we obtain
$A = W^{-1} 
\begin{bmatrix}
A_1 \\
A_2
\end{bmatrix} 
= W^{-1} K U^{-1}$,
and derive
\begin{align*}
& A A_{I}^{\#} A
= W^{-1} K U^{-1} U K^{\#}_I W W^{-1} K U^{-1} 
= W^{-1} K K^{\#}_I K U^{-1} 
= W^{-1} K U^{-1} 
= A, \\
& A_{I}^{\#} A A_{I}^{\#}
= U K^{\#}_I W W^{-1} K U^{-1} U K^{\#}_I W 
= U K^{\#}_I K K^{\#}_I W 
= U K^{\#}_I W 
=  A_{I}^{\#}.
\end{align*}
%To see that $A_{I}^{\#} A$ is integer, 
To conclude the proof, we write 
\begin{align*}
A_{I}^{\#} A 
= U K^{\#}_I W W^{-1} K U^{-1} 
= U K^{\#}_I K U^{-1}
= U \begin{bmatrix}I_r & 0 \\ 0 & 0\end{bmatrix} U^{-1}.
\end{align*}
Since $U$ and $U^{-1}$ are integer, the matrix $A_{I}^{\#} A$ is integer as well.
\end{proof}

\subsection{Characterization of mixed integer solutions to systems of linear equations}

We are now ready to present our polynomial time algorithm for the characterization of mixed integer solutions to systems of linear equations.
%This result extends to the mixed integer setting classic results in the pure integer setting ().
Corresponding results in the pure integer setting can be found, for example, in textbooks \cite{SchBookIP,ConCorZamBook}.

%\note{replace `general solution' with the sets throughout}

\begin{proposition}
\label{prop gen solution}
Let $W \in \Q^{m \times n}$, $w \in \Q^m$, $p \in \bra{0,1,\dots,n}$, and consider the sets
%\begin{align*}
%%\label{eq Hermite}
%\L := \bra{x \in \Z^p \times \R^{n-p} : Wx = w }.
%\end{align*}
\begin{align*}
%\label{eq Hermite}
\A := \bra{x \in \R^n : Wx = w }, \qquad \S := \A \cap \pare{\Z^p \times \R^{n-p}}.
\end{align*}
There is a polynomial time algorithm that checks whether $\S$ is empty or not.
Let $A \in \Q^{m \times p}$, $B \in \Q^{m \times n-p}$ so that $W = \begin{bmatrix} A & B \end{bmatrix}$ and let $n' := n-\rank(W)$ and $p' := p - \rank(W) + \rank(B)$.
If $\S$ is nonempty, the algorithm finds 
a map $\tau : \R^{n'} \to \R^n$
of the form $\tau(x') = \bar x + Mx'$ with $\bar x \in \Z^p \times \Q^{n-p}$ and $M \in \Q^{n \times n'}$ of full rank, such that 
\begin{align*}
\A & = \tau \pare{\R^{n'}} \\
\S & = \tau \pare{\Z^{p'} \times \R^{n'-p'}}.
\end{align*}
% and the last $n-p$ columns of $W$ have all zero elements, then $p' \le p-1$.
Furthermore, if an equality of the form $d^\transp x = \beta$ with $d_{p+1} = \cdots = d_n = 0$ is valid for $\A$, then $p' \le p-1$.
\end{proposition}

\begin{proof}
Let $q := n-p$.
%, and let $A \in \Q^{m \times p}$, $B \in \Q^{m \times q}$ so that $W = \begin{bmatrix} A & B \end{bmatrix}$.
Let $y \in \R^p$, $z \in \R^q$ so that $x = \begin{bmatrix} y \\ z \end{bmatrix}$.
%$c \in \Q^m$, and consider the system
We can then write the set $\S$ in the form
\begin{align*}
%\label{eq Hermite 2}
\S = \bra{\begin{bmatrix} y \\ z \end{bmatrix} \in \Z^p \times \R^q : Ay + Bz = w}.
\end{align*}
%There is a polynomial time algorithm that checks whether the system \eqref{eq Hermite} has a solution.
%If it has a solution, it finds $p' \in \{0,1,\dots,p\}$, $q' \in \{0,1,\dots,q\}$, 
%vectors $\bar y \in \Z^p$, $\bar z \in \Q^q$, and matrices $Q \in \Z^{p \times p'}$, $R \in \Q^{q \times p'}$, $S \in \Q^{q \times q'}$, such that the general solution to \eqref{eq Hermite} is
%\begin{align*}
%y& = \bar x + Q z', \\
%z&= \bar y + R z' + S z',\\
%& \text{where $z' \in \Z^{p'}$ and $z' \in \R^{q'}$.}
%\end{align*}
%Furthermore, $Q,S$ have full column rank and, if $A \neq 0$ and $B = 0$, then $p' \le p-1$.
%
Let $C:=\pare{I_m-B B^{\#}} A \in \R^{m \times p}$ and $d:=\pare{I_m-B B^{\#}} w \in \R^m$.
Theorem~1 in \cite{BowBur74} states that the system 
\begin{align}
\label{eq Hermite 2}
Ay + Bz = w, \qquad 
y \in \Z^p, \qquad 
z \in \R^q.
\end{align}
has a solution if and only if
%\begin{enumerate}
%\item[(i)] 
(i) $C_I^{\#} d$ is integer, and
%\item[(ii)] 
(ii) $C C_I^{\#} d=d$.
%\end{enumerate}
Furthermore, if there exists a solution then the general solution is
\begin{align*}
y&=C_I^{\#} d+\pare{I_p-C_I^{\#} C} y', \\
z&=B^{\#} w-B^{\#} A C_I^{\#} d-B^{\#} A\pare{I_p-C_I^{\#} C} y'+\pare{I_q-B^{\#} B} z', \\
& \text{where $y' \in \Z^p$ and $z' \in \R^q$.}
\end{align*}
We use \cref{lem poly inv} to compute $B^{\#}$, so that we can construct $C$ and $d$.
We use again \cref{lem poly inv} to compute $C_I^{\#}$.
We can then check whether the system \eqref{eq Hermite 2} has a solution by checking (i) and (ii).
If the system has no solutions, then $\S$ is empty and we are done, so assume now that there exists a solution.
\cref{lem poly inv} also finds 
a unimodular matrix $U_B \in \Q^{q \times q}$ such that $B^{\#} B = U_B \begin{bmatrix}I_{r_B} & 0 \\ 0 & 0\end{bmatrix} U_B^{-1}$, where $r_B$ is the rank of $B$, and 
a unimodular matrix $U_C \in \Q^{p \times p}$ such that $C^{\#}_I C = U_C \begin{bmatrix}I_{r_C} & 0 \\ 0 & 0\end{bmatrix} U_C^{-1}$, where $r_C$ is the rank of $C$.
In the formula for a general solution, we replace $B^{\#} B$ and $C^{\#}_I C$ as above, and we replace $I_q$ and $I_p$ with $U_B U_B^{-1}$ and $U_C U_C^{-1}$, respectively. 
We obtain
\begin{align*}
y&=C_I^{\#} d+U_C \pare{I_p- \begin{bmatrix}I_{r_C} & 0 \\ 0 & 0\end{bmatrix} } U_C^{-1} y' \\
& = C_I^{\#} d+U_C \begin{bmatrix}0 & 0 \\ 0 & -I_{p-r_C}\end{bmatrix}  U_C^{-1} y', \\
z&=B^{\#} w-B^{\#} A C_I^{\#} d-B^{\#} A U_C \pare{ I_p- \begin{bmatrix}I_{r_C} & 0 \\ 0 & 0\end{bmatrix} } U_C^{-1} y' + U_B \pare{ I_q - \begin{bmatrix}I_{r_B} & 0 \\ 0 & 0\end{bmatrix} } U_B^{-1} z' \\ 
&=B^{\#} w-B^{\#} A C_I^{\#} d-B^{\#} A U_C \begin{bmatrix}0 & 0 \\ 0 & -I_{p-r_C}\end{bmatrix}  U_C^{-1} y' + U_B \begin{bmatrix}0 & 0 \\ 0 & -I_{q-r_B}\end{bmatrix}  U_B^{-1} z',\\
& \text{where $y' \in \Z^p$ and $z' \in \R^q$.}
\end{align*}
Since $U_B^{-1}$ is invertible, we have $\bra{U_B^{-1} z' : z'  \in \R^q} = \R^q$, and 
since $U_C^{-1}$ is unimodular, we have $\bra{U_C^{-1} y' : y'  \in \Z^p} = \Z^p$.
Thus the general solution is
\begin{align*}
y& = C_I^{\#} d+U_C \begin{bmatrix}0 & 0 \\ 0 & -I_{p-r_C}\end{bmatrix} y', \\
z&=B^{\#} w-B^{\#} A C_I^{\#} d-B^{\#} A U_C \begin{bmatrix}0 & 0 \\ 0 & -I_{p-r_C}\end{bmatrix} y' + U_B \begin{bmatrix}0 & 0 \\ 0 & -I_{q-r_B}\end{bmatrix} z',\\
& \text{where $y' \in \Z^p$ and $z' \in \R^q$.}
\end{align*}
We can now discard the first $r_C$ components of $y'$ and the first $r_B$ components of $z'$ and obtain
\begin{align*}
y& = C_I^{\#} d+U_C \begin{bmatrix}0 \\ -I_{p-r_C}\end{bmatrix} y', \\
z&= B^{\#} w-B^{\#} A C_I^{\#} d - B^{\#} A U_C \begin{bmatrix}0 \\ -I_{p-r_C}\end{bmatrix} y' + U_B \begin{bmatrix}0 \\ -I_{q-r_B}\end{bmatrix} z',\\
& \text{where $y' \in \Z^{p-r_C}$ and $z' \in \R^{q-r_B}$.}
\end{align*}
We define 
\begin{align*}
p' := p-r_C,
\qquad
q' := q-r_B,
\qquad
\bar y := C_I^{\#} d, 
\qquad
\bar z := B^{\#} w-B^{\#} A C_I^{\#} d, \\
R := U_C \begin{bmatrix}0 \\ -I_{p-r_C}\end{bmatrix}, 
\qquad
S := - B^{\#} A U_C \begin{bmatrix}0 \\ -I_{p-r_C}\end{bmatrix},
\qquad
T := U_B \begin{bmatrix}0 \\ -I_{q-r_B}\end{bmatrix}.
\end{align*}
and obtain
\begin{align*}
y& = \bar y + R z', \\
z& = \bar z + S z' + T z',\\
& \text{where $y' \in \Z^{p'}$ and $z' \in \R^{q'}$.}
\end{align*}
We can then write $\S$ in the form
\begin{align*}
\S & = \bra{\bar x + M x' : x' \in \Z^{p'} \times \R^{n'-p'}} = \tau \pare{\Z^{p'} \times \R^{n'-p'}},
\end{align*}
where
\begin{align*}
n' := p' + q',
\qquad
\bar x := \begin{bmatrix} \bar y \\ \bar z \end{bmatrix},
\qquad
M := \begin{bmatrix} R & 0 \\ S & T \end{bmatrix}.
\end{align*}
Note that $R,T$ have full column rank, thus so does $M$.
%Furthermore, if $A \neq 0$ and $B = 0$, then $B^{\#} = 0$, $C = A$, $r_C$ is the rank of $A$ which is at least one, and $p' = p-r_C \le p-1$.
To see that 
\begin{align*}
\A & = \tau \pare{\R^{n'}},
\end{align*}
it suffices to observe that $\spn(\S) = \A$ and $\spn\pare{\Z^{p'} \times \R^{n'-p'}} = \R^{n'}$.

Next, we characterize $p',n'$.
Denote by $r_W$ the rank of $W$.
Since $W = \begin{bmatrix} A & B \end{bmatrix}$, it is well known (see, e.g., theorem 4.1 in \cite{Mey73}) that 
\begin{align*}
r_W = r_B + \rank\pare{\pare{I_m-B B^{\#}} A} = r_B + r_C.
\end{align*}
%We know $q'=q-r_B$, thus 
Hence,
$p'=p-r_C = p-r_W+r_B$ and 
$n'=p'+q'=p-r_W+r_B + q-r_B = n-r_W$.

We now prove the last sentence of the statement.
Assume that an equality of the form $d^\transp x = \beta$ with $d_{p+1} = \cdots = d_n = 0$ is valid for $\A$.
Let $W^1x = w^1$ be obtained from $Wx = w$ by iteratively discarding one linearly dependent equality at the time.
Clearly, $W^1x = w^1$ is a system of linearly independent equalities and $\bra{x \in \R^n : W^1x = w^1} = \A$.
Next, we construct a system of linearly independent equalities, that we denote by $W^2 x = w^2$, such that $\bra{x \in \R^n : W^2 x = w^2} = \A$ and such that the equality $d^\transp x = \beta$ is in the system $W^2 x = w^2$.
If $d^\transp x = \beta$ is in the system $W^1x = w^1$, then we define $W^2 := W^1$, $w^2 := w^1$ and we are done.
Otherwise, $d^\transp$ is a linear combination of the rows of $W^1$.
Let ${d'}^\transp x = \beta'$ be an equality in $W^1x = w^1$ such that the row ${d'}^\transp$  of $W^1$ has a nonzero coefficient in such linear combination.
Let $W^2 x = w^2$ be obtained by replacing, in the system $W^1 x = w^1$, the equality ${d'}^\transp x = \beta'$ with $d^\transp x = \beta$.
We then have $\bra{x \in \R^n : W^2 x = w^2} = \A$ and $d^\transp x = \beta$ is in the system $W^2 x = w^2$.
Let $A^2$, $B^2$ so that $W^2 = \begin{bmatrix} A^2 & B^2 \end{bmatrix}$.
Since $W^2$ has full row rank and $B^2$ has a row of zeros, we have $\rank(B^2) \le \rank(W^2)-1$.
It is simple to check that $\rank(W^2) = \rank(W)$ and $\rank(B^2) = \rank(B)$.
We then obtain $\rank(B) \le \rank(W)-1$, thus $p' = p - \rank(W) + \rank(B) \le p-1$.
\end{proof}

\subsection{Reduction to full-dimensional polyhedra}

We are now ready to present the main result of this section: In \cref{th full dim poly} below, we employ \cref{prop gen solution} to give an algorithm that finds in polynomial time an affine function that maps a non full-dimensional polyhedron to a full-dimensional isomorphic one in a lower-dimensional space, while preserving mixed integer vectors.
%, and it plays a key role in most Lenstra-type algorithms (see, e.g., \cite{Len83}).
Note that it is not hard to design recursive algorithms to achieve the same goal, where one dimension is eliminated at each iteration (see, e.g., \cite{dP23bMPA}). 
While this is a sound approach in fixed dimension, it may not result in a polynomial time algorithm in general dimension, since applying this reduction recursively we might get numbers whose size is exponential in the size of the numbers in the original inequalities.
In this paper, the \emph{size} (also known as \emph{bit size}, or \emph{length}) of rational numbers, vectors, matrices, constraints, and optimization problems, denoted by $\size(\cdot)$, is the standard one in mathematical programming (see, e.g., \cite{SchBookIP,ConCorZamBook}), and is essentially the number of bits required to encode such objects.
%We refer the reader to \cite{SchBookIP} for formal definitions of size.
%We also remark that we explicitly give all constants used by the algorithm, rather than using the big O notation, which is essential to obtain an implementable algorithm.

\begin{theorem}
\label{th full dim poly}
Let $W \in \Q^{m \times n}$, $w \in \Q^m$, $p \in \bra{0,1,\dots,n}$, and consider the sets
\begin{align*}
\P := \bra{x \in \R^n : Wx \le w}, \qquad \S := \P \cap \pare{\Z^p \times \R^{n-p}}.
\end{align*}
There is a polynomial time algorithm that either returns that $\S$ is empty, 
or finds 
$p' \in \bra{0,1,\dots,p}$, $n' \in \bra{p',p'+1,\dots,p'+n-p}$, %$m' \in \bra{0,1,\dots,m}$,
%$W' \in \Q^{m' \times n'}$, $w' \in \Q^{m'}$, 
a map $\tau : \R^{n'} \to \R^n$
of the form $\tau(x') = \bar x + Mx'$, with 
$\bar x \in \Z^p \times \Q^{n-p}$ and $M \in \Q^{n \times n'}$ of full rank, such that
%if we define $W' := WM \in \Q^{m \times n'}$, $w' := w - W \bar x \in \Q^m$,
%the polyhedron $\bra{x' \in \R^{n'} : W'x' \le w'}$ is full-dimensional, and 
the polyhedron 
\begin{align*}
\P' := \bra{x' \in \R^{n'} : WMx' \le w - W \bar x}
\end{align*}
is full-dimensional, and 
%if we define
\begin{align*}
%\P & = \tau \pare{\bra{x' \in \R^{n'} : W'x' \le w'}} \\
\P & = \tau \pare{\P'} \\
\S & = \tau \pare{\P' \cap \pare{\Z^{p'} \times \R^{n'-p'}}}.
%\Z^p \times \R^{n-p} & = \bra{\bar x + M x' : x' \in \Z^{p'} \times \R^{n'-p'}}.
\end{align*}
Furthermore, if an equality of the form $d^\transp x = \beta$ with $d_{p+1} = \cdots = d_n = 0$ is valid for $\P$, then $p' \le p-1$.
\end{theorem}

\begin{proof}
An inequality $d^\transp x \le \beta$ from $Wx \le w$ is called an \emph{implicit equality} (in $Wx \le w$) if $d^\transp x = \beta$ for all $x$ satisfying $Wx \le w$. 
Denote by $W^=x \le w^=$ the system of implicit equalities in $Wx \le w$, and by $W^+x \le w^+$ the system of all other inequalities in $Wx \le w$.
It is well known that we can find in polynomial time the system $W^=x \le w^=$, for example by solving a linear programming problem for each inequality.
If no inequality from $Wx \le w$ is an implicit equality in $Wx \le w$, then we define $M:=I_n$ and $\bar x =0$ and we are done, thus we now assume that the system $W^=x \le w^=$ contains at least one inequality.
We apply \cref{prop gen solution} to the set
\begin{align*}
%\label{eq Hermite inproof}
\A := \bra{x \in \R^n : W^=x = w^= }.
\end{align*}
If \cref{prop gen solution} returns that $\A \cap \pare{\Z^p \times \R^{n-p}} = \emptyset$, then $\S$ is empty and we are done.
Otherwise, \cref{prop gen solution} returns that $\A \cap \pare{\Z^p \times \R^{n-p}} \neq \emptyset$.
Let $A \in \Q^{m \times p}$, $B \in \Q^{m \times n-p}$ so that $W^= = \begin{bmatrix} A & B \end{bmatrix}$ and let $n' := n-\rank(W^=)$ and $p' := p - \rank(W^=) + \rank(B)$.
\cref{prop gen solution} also finds 
a map $\tau : \R^{n'} \to \R^n$
of the form $\tau(x') = \bar x + Mx'$ with $\bar x \in \Z^p \times \Q^{n-p}$ and $M \in \Q^{n \times n'}$ of full rank, such that 
\begin{align*}
\A & = \tau \pare{\R^{n'}} \\
\A \cap \pare{\Z^p \times \R^{n-p}} & = \tau \pare{\Z^{p'} \times \R^{n'-p'}}.
\end{align*}
Since $\P \subseteq \A$ and $\S \subseteq \A \cap \pare{\Z^p \times \R^{n-p}}$, 
%we replace $x$ with $\bar x + M x'$ in $\S$.
%and 
%we drop the constant $\bar x^\transp H \bar x + h^\transp \bar x$ in the objective function.
%and 
%we obtain 
the change of variables $x = \bar x + Mx'$ in the description of $\P$ yields
\begin{align*}
\P & = \tau(\P') \\
\S & = \tau\pare{\P' \cap (\Z^{p'} \times \R^{n'-p'})}.
\end{align*}
%We obtain the instance $I'$ with data 
%$H' := M^\transp H M$,
%${h'}^\transp := h^\transp M + 2 \bar x^\transp HM$,
%Therefore, we set $W' := WM$ and $w' := w - W \bar x.$
%Using Sylvester’s law of inertia, it is simple to check that the number of negative eigenvalues (resp., of zero eigenvalues, of positive eigenvalues) of $H'$ is at most the number of negative eigenvalues, (resp., of zero eigenvalues, of positive eigenvalues) of $H$.
%
Since $\P \cap \bra{x \in \R^n : x^\transp H x + h^\transp x \le \eta} \subseteq \A$, the same change of variables gives
\begin{align*}
& \P \cap \bra{x \in \R^n : x^\transp H x + h^\transp x \le \eta} = \tau\pare{\P' \cap \bra{x' \in \R^{n'} : {x'}^\transp H' x' + {h'}^\transp x' \le \eta'}}.
\end{align*}
It is well known that there is a vector $x \in \R^n$ satisfying $W^=x = w^=$, $W^+x < w^+$ (see, e.g., \cite{SchBookIP}), thus $\P'$ is full-dimensional.
%
%Since the definition of~$\epsilon$-approximate solution is preserved under changes of variables and translations of the objective function, for $\epsilon \in [0,1],$ $x^\diamond$ is an $\epsilon$-approximate solution to $I'$ if and only if $\bar x + M x^\diamond$ is an $\epsilon$-approximate solution to $I$.
%This completes the first part of the proof.
%

If we assume that an equality of the form $d^\transp x = \beta$ with $d_{p+1} = \cdots = d_n = 0$ is valid for $\P$, then it is also valid for $\A$, and \cref{prop gen solution} implies $p' \le p-1$.
\end{proof}

\section{Convex Quadratic Sets}
\label{sec cqs}

In this section we 
%provide an in-depth 
study convex quadratic sets.
A \emph{convex quadratic set} is a set of the form
\begin{align*}
\sQ = \{x \in \R^n : Wx \le w, \  x^\transp H x + h^\transp x \le \eta \},
\end{align*}
where $W \in \Q^{m \times n}$, $w \in \Q^m$, $H \in PSD^n(\Q)$, $h \in \Q^n$, and $\eta \in \Q$.
The main objective is to present and prove \cref{th full-dim conv quad}, 
%which is our extension of \cref{th full dim poly} to convex quadratic sets.
which allows us to find 
%It we show how we can find 
in polynomial time an affine function that maps a non full-dimensional convex quadratic set to a full-dimensional isomorphic one in a lower-dimensional space, while preserving mixed integer vectors.

%\section{Reduction to mixed integer points in full-dimensional sets}

%\subsection{Characterization of full-dimensionality}
\subsection{Characterization of full-dimensional convex quadratic sets}

The first goal of this section is to categorize convex quadratic sets into three types, and to obtain structural and algorithmic results for each type.
% obtain structural and algorithmic results for three different types of convex quadratic sets.
These types of convex quadratic sets are considered separately in \cref{lem low dim 1,lem low dim 2,lem full-dim} below.
%These three lemmas will then be used to give a characterization of full-dimensional convex quadratic sets that can be checked in polynomial time.

%\note{From lemmas to propositions?}

\begin{lemma}
%[Non full-dimensional convex quadratic sets: polyhedral case]
\label{lem low dim 1}
Let $\sQ$ be a convex quadratic set
\begin{align*}
\sQ = \{x \in \R^n : Wx \le w , \ x^\transp H x + h^\transp x \le \eta \},
\end{align*}
where $W \in \Q^{m \times n}$, $w \in \Q^m$, $H \in PSD^n(\Q)$, $h \in \Q^n$, and $\eta \in \Q$.
Assume 
\begin{align*}
\bar \eta & := \min\{x^\transp H x + h^\transp x : Wx \le w\} = \eta \\
\tilde \eta & := \min\{x^\transp H x + h^\transp x : x \in \R^n\} = \eta.
\end{align*}
Then $\sQ$ is nonempty and not full-dimensional.
Furthermore, there is a polynomial time algorithm that finds an affine subspace $\A$ of  $\R^n$ such that $\sQ = \{x \in \R^n : Wx \le w \} \cap \A$.
\end{lemma}

\begin{proof}
The assumption $\bar \eta = \eta$ implies that $\sQ$ is nonempty.
%It is well known that there is a polynomial time algorithm that detects whether $\P$ is full-dimensional;
%% or not; 
%If it is not, it also returns an affine subspace $\A$ of $\R^n$ (the affine hull of $\P$) such that $\P \cap \A$ is full-dimensional. (see, e.g., \cite{FreRouTod85,SchBookIP}).
%If $\P$ is not full-dimensional, then also $\sQ$ is not full-dimensional.
%Denote by $n'$ the dimension of $\A$.
%We construct a map $\pi : \R^{n'} \to \A$ of the form $\pi(y) = \bar x + My$ with $\bar x \in \A$ and $M$ of rank $n'$ in $\Q^{n\times n'}$ such that $\pi(\R^{n'}) = \A$.
%In the remainder of the proof, we assume that $\P$ is full-dimensional.
%\note{Expand, need to map $\P$ and $\sQ$.}
%
%Let $q(x) := x^\transp H x + h^\transp x$, solve the convex quadratic programming problem
%\begin{align}
%\label{eq some cont min inproof}
%\min\{q(x) : x \in \P\}
%\end{align}
%with the Kozlov-Tarasov-Khachiyan algorithm \cite{KozTarKha81},
%and denote by $\bar \eta \in \{-\infty\} \cup \Q$ its minimal value.
%Since $\P$ is bounded, we find an optimal solution $\bar x$.
%Let $\bar \eta := q(\bar x)$ be the minimal value.
%If $\bar \eta > \eta$, then $\sQ$ is the empty set and we are done, thus in the remainder of the proof we assume $\bar \eta \le \eta$.
%In the next claim we consider the case $\bar \eta = \eta$.
%
%\begin{claim}
%\label{claim find hyperplane}
%If $\bar \eta = \eta$, we find a hyperplane $\H$ containing $\sQ$.
%\end{claim}
%
%\note{point out wlog $H,h$ are not both zero}
Let $q(x) := x^\transp H x + h^\transp x$.
%\begin{cpf}
%Since $\bar \eta$ is finite, the Kozlov-Tarasov-Khachiyan algorithm finds a vector $\bar x \in \P$ with $q(\bar x) = \bar \eta$.
%
%If $\bar \eta > \eta$, then $\sQ$ is the empty set, we return an arbitrary hyperplane $\H$, and we are done.
%Thus, in the remained of the proof we assume $\bar \eta = \eta$.
%Note that the assumption $\tilde \eta = \eta$ implies that $H$ is not the zero matrix.
%We solve $\min\{q(x) : x \in \R^n\}$, using the Kozlov-Tarasov-Khachiyan algorithm, and we let $\tilde \eta \in \{-\infty\} \cup \Q$ be its minimal value.
%We then consider separately the two cases $\tilde \eta = \eta$ and $\tilde \eta < \eta$.
%We first consider the case $\tilde \eta = \eta$.
Note that $H$ is nonzero since $\tilde \eta = \eta$ is finite.
The set $\{x \in \R^n : q(x) = \tilde \eta\}$ is then the set of minima of the convex function $q(x)$ over $\R^n$.
Hence,
\begin{align*}
\{x \in \R^n : q(x) \le \eta\} 
%& = \{x \in \R^n : q(x) = \eta\} \\
& = \{x \in \R^n : q(x) = \tilde \eta\} \\
& = \{x \in \R^n : \nabla q (x) = 0\} \\
& = \{x \in \R^n : 2H x+h=0\}.
\end{align*}
The statement follows by defining $\A := \{x \in \R^n : 2H x+h=0\}.$
The set $\sQ$ is not full-dimensional, since it is contained in $\A$.
\end{proof}

\begin{lemma}
%[Non full-dimensional convex quadratic sets: non-polyhedral case]
\label{lem low dim 2}
Let $\sQ$ be a convex quadratic set
\begin{align*}
\sQ = \{x \in \R^n : Wx \le w , \ x^\transp H x + h^\transp x \le \eta \},
\end{align*}
where $W \in \Q^{m \times n}$, $w \in \Q^m$, $H \in PSD^n(\Q)$, $h \in \Q^n$, and $\eta \in \Q$.
Assume $\P := \{x \in \R^n : Wx \le w \}$ is full-dimensional and 
\begin{align*}
\bar \eta & := \min\{x^\transp H x + h^\transp x : Wx \le w\} = \eta \\
\tilde \eta & := \min\{x^\transp H x + h^\transp x : x \in \R^n\} < \eta.
\end{align*}
Then $\sQ$ is nonempty and not full-dimensional.
Furthermore, there is a polynomial time algorithm that 
%there is a polynomial time algorithm that 
%detects whether $\sQ$ is empty.
% or not.
%If $\sQ = \emptyset$, it returns so.
%Otherwise, if $\sQ \neq \emptyset$, it 
finds a proper face $\F$ of $\P$ (obtained from the description of $\P$ by setting some inequality constraints to equality) such that $\sQ$ is contained in $\F$.
%$\H = \{x \in \R^n : c^\transp x = \gamma\}$ 
%The running time of the algorithm is polynomial in the size of $W,w,H,h,\eta$.
\end{lemma}

\begin{proof}
The assumption $\bar \eta = \eta$ implies that $\sQ$ is nonempty.
%It is well known that there is a polynomial time algorithm that detects whether $\P$ is full-dimensional;
%% or not; 
%If it is not, it also returns an affine subspace $\A$ of $\R^n$ (the affine hull of $\P$) such that $\P \cap \A$ is full-dimensional. (see, e.g., \cite{FreRouTod85,SchBookIP}).
%If $\P$ is not full-dimensional, then also $\sQ$ is not full-dimensional.
%Denote by $n'$ the dimension of $\A$.
%We construct a map $\pi : \R^{n'} \to \A$ of the form $\pi(y) = \bar x + My$ with $\bar x \in \A$ and $M$ of rank $n'$ in $\Q^{n\times n'}$ such that $\pi(\R^{n'}) = \A$.
%In the remainder of the proof, we assume that $\P$ is full-dimensional.
%\note{Expand, need to map $\P$ and $\sQ$.}
%
%Let $q(x) := x^\transp H x + h^\transp x$, solve the convex quadratic programming problem
%\begin{align}
%\label{eq some cont min inproof}
%\min\{q(x) : x \in \P\}
%\end{align}
%with the Kozlov-Tarasov-Khachiyan algorithm \cite{KozTarKha81},
%and denote by $\bar \eta \in \{-\infty\} \cup \Q$ its minimal value.
%%Since $\P$ is bounded, we find an optimal solution $\bar x$.
%%Let $\bar \eta := q(\bar x)$ be the minimal value.
%If $\bar \eta > \eta$, then $\sQ$ is the empty set and we are done, thus in the remainder of the proof we assume $\bar \eta \le \eta$.
%In the next claim we consider the case $\bar \eta = \eta$.
%
%\begin{claim}
%\label{claim find hyperplane}
%If $\bar \eta = \eta$, we find a hyperplane $\H$ containing $\sQ$.
%\end{claim}
%
%\note{point out wlog $H,h$ are not both zero}
%
%\begin{cpf}
Let $q(x) := x^\transp H x + h^\transp x$.
%\begin{cpf}
First, we show how that we can find a supporting hyperplane $\H$ of $\P$ that contains $\sQ$.
If $H=0$, then we can simply set $\H := \{x \in \R^n : h^\transp x = \eta\}$, thus we now assume $H$ nonzero.
Since $\bar \eta$ is finite, the Kozlov-Tarasov-Khachiyan algorithm finds a vector $\bar x \in \P$ with $q(\bar x) = \bar \eta$.
%If $\bar \eta > \eta$, then $\sQ$ is the empty set, we return an arbitrary hyperplane $\H$, and we are done.
%Thus, in the remained of the proof we assume $\bar \eta = \eta$.
%If $H$ is the zero matrix, the set $\sQ$ is contained in the hyperplane $\H = \{x \in \R^n : h^\transp x = \eta \}$, and we are done.
%Thus, in the remainder of the proof, we assume that $H$ is not the zero matrix.
%
%We solve $\min\{q(x) : x \in \R^n\}$, using the Kozlov-Tarasov-Khachiyan algorithm, and we let $\tilde \eta \in \{-\infty\} \cup \Q$ be its minimal value.
%We then consider separately the two cases $\tilde \eta = \eta$ and $\tilde \eta < \eta$.
%
%We first consider the case $\tilde \eta = \eta$.
%The set $\{x \in \R^n : q(x) = \eta\}$ is then the set of minima of the convex function $q(x)$ over $\R^n$.
%Hence,
%\begin{align*}
%\{x \in \R^n : q(x) \le \eta\} 
%= \{x \in \R^n : q(x) = \eta\} 
%= \{x \in \R^n : \nabla q (x) = 0\}
%= \{x \in \R^n : 2H x+h=0\}.
%\end{align*}
%%For every such minima $\tilde x$, we \note{this should equal the set of minima} have $\nabla q (\tilde x) = 2H \tilde x+h=0$.
%%Since $H$ is not the zero matrix, at least one of the $n$ equalities in $2H \tilde x = - h$ has a nonzero left-hand side.
%%We denote one such equality by $c^\transp x = \gamma$, and so the hyperplane $\H = \{x \in \R^n : c^\transp x = \gamma \}$ contains $\{x \in \R^n : q(x) = \eta\}$ and therefore contains $\sQ$.
%
%In the remainder of the proof, we consider the case $\tilde \eta < \eta$. \note{This case is still open!}
%The set $\{x \in \R^n : q(x) < \eta\}$ is convex, nonempty, and disjoint from $\P$.
The hyperplane tangent to $\{x \in \R^n : q(x) \le \eta\}$ in $\bar x$ is 
\begin{align*}
\H := \{x \in \R^n : \nabla q(\bar x)^\transp (x - \bar x) = 0 \}
= \{x \in \R^n : (2H\bar x)^\transp x = (2H\bar x)^\transp \bar x \}.
\end{align*}
Clearly, the inequality $(2H\bar x)^\transp x \le (2H\bar x)^\transp \bar x$ is valid for $\{x \in \R^n : q(x) \le \eta\}$.
On the other hand, since $\bar x \in \P$, $\tilde \eta < \eta$, and $\P$ is convex, $(2H\bar x)^\transp x \ge (2H\bar x)^\transp \bar x$ is valid for $\P$.
Therefore, $\sQ$ is contained in $\H$.

We define the face $\F$ of $\P$ in the statement as 
\begin{align*}
\F := \P \cap \H.
\end{align*}
Since $\P$ is full-dimensional, $\F$ is a proper face of $\P$.
It is well known that a description of $\F$ can be obtained from the description of $\P$ by setting some inequality constraints to equality (see theorem 3.24 in \cite{ConCorZamBook}).
These inequalities can be identified by solving linear programming problems on $\F$.
%From the hyperplane separation theorem, there exists a hyperplane $\H = \{x \in \R^n : c^\transp x = \gamma \}$ such that $c^\transp x \ge \gamma$ is valid for $\P$ and $c^\transp x < \gamma$ is valid for $\{x \in \R^n : q(x) < \eta\}$.
%The only such hyperplane \note{true? explain? needed?} is $\H = \{x \in \R^n : (2H\bar x)^\transp x = (2H\bar x)^\transp \bar x\}$, where $\bar x$ is an optimal solution to $\min\{q(x) : x \in \P\}$.
%In fact, from the convexity of $q$, we have that 
%\end{cpf}
\end{proof}

\begin{lemma}
%[Full-dimensional convex quadratic sets]
\label{lem full-dim}
Let $\sQ$ be a convex quadratic set
\begin{align*}
\sQ = \{x \in \R^n : Wx \le w , \ x^\transp H x + h^\transp x \le \eta \},
\end{align*}
where $W \in \Q^{m \times n}$, $w \in \Q^m$, $H \in PSD^n(\Q)$, $h \in \Q^n$, and $\eta \in \Q$.
Assume $\P := \{x \in \R^n : Wx \le w \}$ is full-dimensional and 
\begin{align*}
\bar \eta := \min\{x^\transp H x + h^\transp x : Wx \le w\} < \eta.
\end{align*}
Then $\sQ$ is full-dimensional.
Furthermore, there is a polynomial time algorithm that finds a full-dimensional polytope contained in $\sQ$.
\end{lemma}

\begin{proof}
Let $q(x) := x^\transp H x + h^\transp x$.
%Use the Kozlov-Tarasov-Khachiyan algorithm to solve
%\begin{align*}
%%\label{eq some cont min inproof}
%\bar \eta := \min\{q(x) : x \in \P\}.
%\end{align*}
%%and let $\bar x$ be an optimal solution.
%
%\begin{claim}
%\label{cla smaller point}
First, we claim that 
there is a polynomial time algorithm that finds $\bar x \in \P$ with $q(\bar x) < \eta$.
%\end{claim}
%\begin{cpf}
If $\bar \eta \in \Q$, we let $\bar x$ be an optimal solution to 
$
\min\{q(x) : x \in \P\}
$
found by the Kozlov-Tarasov-Khachiyan algorithm and we have $q(\bar x) = \bar \eta < \eta$.
Thus, we now assume $\bar \eta = -\infty$.
Consider the convex quadratic set $\sQ' := \{x \in \P : x^\transp H x + h^\transp x \le \eta -1 \}$.
We can assume without loss of generality that the data defining $\sQ'$ is integer.
This can be done by scaling, at the expense of multiplying the size of the system defining $\sQ'$
by $n^2$ (see, e.g., remark~1.1 in \cite{ConCorZamBook}).
Denote by $s$ the size of the obtained system with integer data defining $\sQ'$.
%, once coefficients have been scaled to be integer (see remark 1.1 in \cite{ConCorZamBook}).
Since $\bar \eta = -\infty$, the set $\sQ'$ is nonempty, and it follows from theorem~1 in \cite{Kha83} that there exists $x^* \in \sQ'$ that satisfies 
\begin{align*}
\norm{x^*} 
\le (2n 2^s)^{2^4 n}
\le (2s 2^s)^{2^4 s}
\le (2^{2s})^{2^4 s}
= 2^{2^5s^2}.
\end{align*}
%where we used $n \le s$ and $2s \le 2^s$.
Therefore, $x^*$ satisfies the $2n$ inequalities $-2^{2^5 s^2} \le x_i \le 2^{2^5 s^2}$, for $i = 1,2,\dots,n$.
We solve, with the Kozlov-Tarasov-Khachiyan algorithm, the convex quadratic programming problem
\begin{align}
\label{eq some cont min inproof v2}
\min\bra{q(x) : x \in \P, \ -2^{2^5 s^2} \le x_i \le 2^{2^5 s^2}, \forall i = 1,2,\dots,n}.
\end{align}
Since the feasible region is bounded, the algorithm returns an optimal solution, which we denote by $\bar x$.
We then have $q(\bar x) \le \eta - 1 < \eta$.
%\end{cpf}
This concludes the proof of our claim that there is a polynomial time algorithm that finds $\bar x \in \P$ with $q(\bar x) < \eta$.

%To conclude the proof, it suffices to show the next claim.

%\begin{claim}
%In the remainder of the proof, we show that there is a polynomial time algorithm that finds a full-dimensional polytope $\F$ contained in $\sQ$.
%\end{claim}
%
%\begin{cpf}
%Let $\bar x \in \P$ with $q(\bar x) < \eta$ from \cref{cla smaller point}.
We define $\alpha := 2^{\size(H,h)}$ and $\beta := 2^{\size(\bar x)}+1$
and we apply a classic result on Lipschitz continuity of a polynomial on a box,
%since $\delta \le 1$, 
%we then obtain 
%$x, \bar x \in [-\beta,\beta]^n$, for every $x \in \C$.
%We then apply the Lipschitz continuity result, 
which is lemma~3.1 in \cite{BiedPHil23MPB} (with $d:=2$, $H:=\alpha$, $M:=\beta$).
We obtain
\begin{align}
\label{eq Lip 1}
\abs{q(y) - q(z)} 
\leq 2 \alpha \beta n (n+2) \infnorm{y - z} && \forall y,z \in [-\beta,\beta]^n.
\end{align}
%We now show that $q(x) \le \eta$ is valid for $\C$.
%
Now define $\delta := \min\{1, \ (\eta - \bar \eta) / (2 \alpha \beta n (n+2))\}$ and 
\begin{align*}
\C := \{ x \in \R^n : \infnorm{x-\bar x} \le \delta \} = \{ x \in \R^n : \bar x - \delta \le x \le \bar x + \delta \}.
\end{align*}
Note that the size of $\delta$ is polynomial in the size of $W,w,H,h,\eta$.
Clearly $\bar x \in [-2^{\size(\bar x)}, 2^{\size(\bar x)}]^n \subseteq [-\beta,\beta]^n$.
Furthermore, since 
$\delta \le 1$, we have $\C \subseteq [-\beta,\beta]^n$.
From \eqref{eq Lip 1} we then obtain 
\begin{align*}
\abs{q(x) - q(\bar x)} 
\leq 2 \alpha \beta n (n+2) \infnorm{x - \bar x} && \forall x \in \C.
\end{align*}
Hence, for every $x \in \C$,
\begin{align*}
q(x) - q(\bar x)
\le \abs{q(x) - q(\bar x)} 
\leq 2 \alpha \beta n (n+2) \infnorm{x - \bar x}
\le 2 \alpha \beta n (n+2) \delta
\le \eta - \bar \eta,
\end{align*}
where in the last inequality we used the definition of $\delta$.
%Here, the second inequality hods from \cref{lem lipschitz}.
We have thereby shown that $q(x) \le \eta$ holds for every $x \in \C$.
% and so $x \in \sQ$ and $\C \subseteq \sQ$.

Consider now the polytope $\P \cap \C$, which is contained in $\sQ$.
%Next, we show that $\P \cap \C$ is full-dimensional.
Since $\P$ is full-dimensional and $\bar x \in P$, for each $\epsilon > 0$ the set $\P \cap \{ x \in \R^n : \infnorm{x-\bar x} \le \epsilon \}$ is full-dimensional, thus so is $\P \cap \C$.
%\end{cpf}
\end{proof}

Leveraging on \cref{lem low dim 1,lem low dim 2,lem full-dim}, we can now 
%\note{Why now $\Q$ instead of $\Z$ in this proposition?}
give a characterization of full-dimensional convex quadratic sets
%These three lemmas will then be used to give a characterization of full-dimensional convex quadratic sets 
that can be checked in polynomial time.

\begin{proposition}
%[Characterization of full-dimensional convex quadratic sets]
\label{prop full-dim}
Let $\sQ$ be a convex quadratic set
\begin{align*}
\sQ = \{x \in \R^n : Wx \le w , \ x^\transp H x + h^\transp x \le \eta \},
\end{align*}
where $W \in \Q^{m \times n}$, $w \in \Q^m$, $H \in PSD^n(\Q)$, $h \in \Q^n$, and $\eta \in \Q$.
Then $\sQ$ is full-dimensional if and only if $\P := \{x \in \R^n : Wx \le w\}$ is full-dimensional and
\begin{align*}
\min\{x^\transp H x + h^\transp x : Wx \le w\} < \eta.
\end{align*}
Furthermore, there is a polynomial time algorithm that detects whether $\sQ$ is full-dimensional or not. 
%Let
%\begin{align*}
%\tilde \eta & := \min\{x^\transp H x + h^\transp x : x \in \R^n\}.
%\end{align*}
%If $\bar \eta = \tilde \eta = \eta$, then $\sQ$ is nonempty and not full-dimensional.
%Furthermore, there is a polynomial time algorithm that finds an affine subspace $\A$ of  $\R^n$ such that $\sQ = \{x \in \R^n : Wx \le w \} \cap \A$.
%
%
%
%If $\sQ$ is full-dimensional, there is a polynomial time algorithm that finds a full-dimensional polytope $\F$ contained in $\sQ$.
\end{proposition}

\begin{proof}
If $\P$ is not full-dimensional, then clearly $\sQ$ is not full-dimensional.
Thus, in the remainder of the proof, we assume that $\P$ is full-dimensional.
Let
\begin{align*}
& \bar \eta := \min\{x^\transp H x + h^\transp x : Wx \le w\}, \\
& \tilde \eta := \min\{x^\transp H x + h^\transp x : x \in \R^n\}.
\end{align*}
If $\bar \eta < \eta$, \cref{lem full-dim} implies that $\sQ$ is full-dimensional.
%, and there is a polynomial time algorithm that finds a full-dimensional polytope $\F$ contained in $\sQ$.
%Note that 
%, and note that $\tilde \eta \in \{-\infty\} \cup \Q$.
%Thus, in the remainder of the proof we show that 
%Let $q(x) := x^\transp H x + h^\transp x$, solve the convex quadratic programming problem
If $\bar \eta > \eta$, then $\sQ$ is the emptyset, thus we now assume $\bar \eta = \eta$.
%
%with the Kozlov-Tarasov-Khachiyan algorithm \cite{KozTarKha81},
%and note that $\tilde \eta \in \{-\infty\} \cup \Q$.
%Since $\P$ is bounded, we find an optimal solution $\bar x$.
%Let $\bar \eta := q(\bar x)$ be the minimal value.
%In the next claim we consider the case $\bar \eta \ge \eta$.
Clearly, $\tilde \eta \le \bar \eta = \eta$.
Then \cref{lem low dim 1,lem low dim 2} imply that $\sQ$ is not full-dimensional (\cref{lem low dim 1} in the case $\tilde \eta = \eta$ and \cref{lem low dim 2} in the case $\tilde \eta < \eta$). 

To see that there is a polynomial time algorithm that detects whether $\sQ$ is full-dimensional or not, it suffices to observe that: (i) there is a polynomial time algorithm that detects whether $\P$ is full-dimensional or not (see, e.g., \cite{SchBookIP}); (ii) $\bar \eta$ can be found with the Kozlov-Tarasov-Khachiyan algorithm \cite{KozTarKha81}.
\end{proof}

\subsection{Reduction to full-dimensional convex quadratic sets}

In the next result, we provide our extension of \cref{th full dim poly} to convex quadratic sets.
Namely, we show that we can find in polynomial time an affine function that maps a non full-dimensional convex quadratic set to a full-dimensional isomorphic one in a lower-dimensional space, while preserving mixed integer vectors.
To prove this result we use \cref{th full dim poly} and \cref{lem low dim 1,lem low dim 2,lem full-dim}.

\begin{theorem}
\label{th full-dim conv quad}
Let $W \in \Q^{m \times n}$, $w \in \Q^m$, $H \in PSD^n(\Q)$, $h \in \Q^n$, $\eta \in \Q$, $p \in \bra{0,1,\dots,n}$, and consider the sets
%Let $\sQ$ be a convex quadratic set
\begin{align*}
\sQ:= \bra{x \in \R^n : Wx \le w , \ x^\transp H x + h^\transp x \le \eta }, \qquad \S := \sQ \cap \pare{\Z^p \times \R^{n-p}}.
\end{align*}
There is a polynomial time algorithm that either returns that $\S$ is empty, 
or finds 
$p' \in \bra{0,1,\dots,p}$, $n' \in \bra{p',p'+1,\dots,p'+n-p}$, %$m' \in \bra{0,1,\dots,m}$,
%$W' \in \Q^{m' \times n'}$, $w' \in \Q^{m'}$, 
a map $\tau : \R^{n'} \to \R^n$
of the form $\tau(x') = \bar x + Mx'$, with 
$\bar x \in \Z^p \times \Q^{n-p}$ and $M \in \Q^{n \times n'}$ of full rank, such that, 
if we define $W' := WM \in \Q^{m \times n'}$, $w' := w - W \bar x \in \Q^m$,
$H' := M^\transp H M \in PSD^{n'}(\Q)$,
$h' := 2M^\transp H^\transp \bar x + M^\transp h \in \Q^{n'}$, 
$\eta' := \eta - \bar x^\transp H \bar x + h^\transp \bar x \in \Q$,
the convex quadratic set 
%\begin{align*}
%\sQ' := \bra{x' \in \R^{n'} : WMx' \le w - W \bar x, \ {x'}^\transp (M^\transp H M) x' + {(2M^\transp H^\transp \bar x + M^\transp h)}^\transp x' \le (\eta - \bar x^\transp H \bar x + h^\transp \bar x)}
%\end{align*} 
\begin{align*}
\sQ' := \bra{x' \in \R^{n'} : W'x' \le w', \ {x'}^\transp H' x' + {h'}^\transp x' \le \eta'}
\end{align*} 
is full-dimensional, and 
\begin{align*}
%\bra{x \in \R^n : Wx \le w} & = \bra{\bar x + M x' : W'x' \le w'}, \\
%\sQ & = \tau\pare{x' \in \R^{n'} : \bra{{x'}^\transp H' x' + {h'}^\transp x' \le \eta'}}, \\
%\Z^p \times \R^{n-p} & = \tau(\Z^{p'} \times \R^{n'-p'}).
\sQ & = \tau \pare{\sQ'} \\
\S & = \tau \pare{\sQ' \cap \pare{\Z^{p'} \times \R^{n'-p'}}}.
\end{align*}
%Furthermore, if an equality of the form $d^\transp x = \beta$ with $d_{p+1} = \cdots = d_n = 0$ is valid for the polyhedron $\bra{x \in \R^n : Wx \le w}$, \note{check!} then $p' \le p-1$.
\end{theorem}

\begin{proof}
If $H=0$, the result follows by applying \cref{th full dim poly} to the polyhedron 
$
\{x \in \R^n : Wx \le w, \ h^\transp x \le \eta \},
$
thus we assume $H$ nonzero.

The algorithm that we present is recursive.
Each iteration starts with a face $\F$ of the polyhedron
\begin{align*}
\P := \{x \in \R^n : Wx \le w \}
\end{align*}
such that 
\begin{align*}
\sQ = \bra{x \in \F : x^\transp H x + h^\transp x \le \eta }.
\end{align*}
In the first iteration we have $\F = \P$, and in each iteration the dimension of $\F$ strictly decreases.
The algebraic description of $\F$ is obtained from $Wx \le w$ by setting some inequality constraints to equality. Next, we describe an iteration of the algorithm.

%Define 
%and the convex quadratic set
%\begin{align*}
%\sQ = \{x \in \R^n : Wx \le w , \ q(x) \le \eta \}.
%\end{align*}
First, we apply \cref{th full dim poly} to $\F$, in order to consider an ``equivalent'' full-dimensional $\F^\circ$. We detail this reduction.
If \cref{th full dim poly} returns that $\F \cap \pare{\Z^p \times \R^{n-p}}$ is empty, then $\S$ is empty and we are done.
Otherwise, \cref{th full dim poly} finds 
%$p^\circ \in \bra{0,1,\dots,p}$, $n^\circ \in \bra{p^\circ,p^\circ+1,\dots,p^\circ+n-p}$, $m^\circ \in \bra{0,1,\dots,m}$,
$p^\circ$, $n^\circ$,
%$W^\circ \in \Q^{m^\circ \times n^\circ}$, $w^\circ \in \Q^{m^\circ}$, 
and an affine map $\tau^\circ : \R^{n^\circ} \to \R^n$
%of the form $\tau^\circ(x^\circ) = \bar x^\circ + M^\circ x^\circ$, with 
%$\bar x^\circ \in \Z^p \times \Q^{n-p}$ and $M^\circ \in \Q^{n \times n^\circ}$ of full rank, 
such that
%if we define $W^\circ := WM^\circ \in \Q^{m \times n^\circ}$, $w^\circ := w - W \bar x^\circ \in \Q^m$,
%the polyhedron $\bra{x^\circ \in \R^{n^\circ} : W^\circx^\circ \le w^\circ}$ is full-dimensional, and 
%the polyhedron 
%\begin{align*}
%\F^\circ := \bra{x^\circ \in \R^{n^\circ} : WM^\circ x^\circ \le w - W \bar x^\circ}
%\end{align*}
the preimage $\F^\circ$ of $\F$
is full-dimensional, and 
%if we define
\begin{align*}
%\P & = \tau \pare{\bra{x^\circ \in \R^{n^\circ} : W^\circx^\circ \le w^\circ}} \\
\F & = \tau^\circ \pare{\F^\circ} \\
\F \cap \pare{\Z^p \times \R^{n-p}} & = \tau^\circ \pare{\F^\circ \cap \pare{\Z^{p^\circ} \times \R^{n^\circ-p^\circ}}}.
%\Z^p \times \R^{n-p} & = \bra{\bar x^\circ + M^\circ x^\circ : x^\circ \in \Z^{p^\circ} \times \R^{n^\circ-p^\circ}}.
\end{align*}
%We define $H^\circ := {M^\circ}^\transp H M^\circ \in \Q^{n^\circ \times n^\circ}$,
%$h^\circ := 2{M^\circ}^\transp H^\transp \bar x^\circ + {M^\circ}^\transp h \in \Q^{n^\circ}$, 
%$\eta^\circ := \eta - {\bar x^\circ}^\transp H \bar x^\circ + h^\transp \bar x^\circ \in \Q$ and obtain
Let $H^\circ\in \Q^{n^\circ \times n^\circ}$,
$h^\circ \in \Q^{n^\circ}$, 
$\eta^\circ \in \Q$, and 
\begin{align*}
& \sQ^\circ
:= \bra{x^\circ \in \F^\circ : {x^\circ}^\transp H^\circ x^\circ + {h^\circ}^\transp x^\circ \le \eta^\circ},
\end{align*}
so that
\begin{align*}
%& \bra{x \in \F : x^\transp H x + h^\transp x \le \eta} 
\sQ
%= \tau^\circ \pare{\bra{x^\circ \in \F^\circ : {x^\circ}^\transp H^\circ x^\circ + {h^\circ}^\transp x^\circ \le \eta^\circ}}.
= \tau^\circ \pare{\sQ^\circ}.
\end{align*}
Clearly, we also have 
\begin{align*}
%\sQ & = \tau^\circ \pare{\sQ^\circ} \\
\S & = \tau^\circ \pare{\sQ^\circ \cap \pare{\Z^{p^\circ} \times \R^{n^\circ-p^\circ}}}.
\end{align*}

% we can assume without loss of generality that $\P$ is full-dimensional.

For ease of notation, let $q^\circ(x^\circ) := {x^\circ}^\transp H^\circ x^\circ + {h^\circ}^\transp x^\circ$.
Solve the convex quadratic programming problem
\begin{align*}
%\label{eq some cont min inproof}
\bar \eta := \min\{q^\circ(x^\circ) : x \in \F^\circ\}
\end{align*}
with the Kozlov-Tarasov-Khachiyan algorithm \cite{KozTarKha81} and note that $\bar \eta \in \{-\infty\} \cup \Q$.
%Since $\P$ is bounded, we find an optimal solution $\bar x$.
%Let $\bar \eta := q(\bar x)$ be the minimal value.
If $\bar \eta > \eta$, then $\sQ^\circ = \emptyset$, hence $\S = \emptyset$ and we are done.
If $\bar \eta < \eta$, then \cref{lem full-dim} implies that $\sQ^\circ$ is full-dimensional.
Then we define $\tau:=\tau^\circ$ and we are done.
Thus in the remainder of the iteration we assume $\bar \eta = \eta$.
%In the next claim we consider the case $\bar \eta = \eta$.

Next, solve the convex quadratic programming problem
\begin{align*}
\tilde \eta := \min\{q^\circ(x^\circ) : x^\circ \in \R^{n^\circ}\}
\end{align*}
with the Kozlov-Tarasov-Khachiyan algorithm, and note that $\tilde \eta \in \{-\infty\} \cup \Q$.
%We then consider separately the two cases $\tilde \eta = \eta$ and $\tilde \eta < \eta$.
Clearly, $\tilde \eta \le \eta$.
In the case $\tilde \eta = \eta$ we employ \cref{lem low dim 1} and obtain an affine subspace $\A$ of  $\R^{n^\circ}$ such that $\sQ^\circ = \F^\circ \cap \A$.
We then apply \cref{th full dim poly} to the polyhedron $\sQ^\circ$.
If \cref{th full dim poly} returns that $\sQ^\circ \cap \pare{\Z^{p^\circ} \times \R^{n^\circ-p^\circ}}$ is empty, then $\S$ is empty and we are done.
Otherwise, \cref{th full dim poly} finds an affine map $\tau^\bullet$.
We then return the affine map $\tau := \tau^\circ \circ \tau^\bullet$ and we are done.
Hence, we now consider the case $\tilde \eta < \eta$.
We employ \cref{lem low dim 2} and find a proper face $\G^\circ$ of $\F^\circ$ (obtained from the description of $\F^\circ$ by setting some inequality constraints to equality) that contains $\sQ^\circ$.
Note that $\tau^\circ(\G^\circ)$ is a proper face of $\F$, that we denote by $\G$.
Since faces of faces of a polyhedron are again faces of a polyhedron, 
%we can, equivalently, recursively apply the algorithm described so far to the face 
$\G$ is also a face of $\P$.
We then recursively apply the algorithm described so far to the face $\G$ of $\P$.

% $\F$ of $\P$ are  the corresponding face of the .
At each iteration, the dimension of the face of $\P$ considered decreases by at least one, thus the algorithm performs a polynomial number of arithmetic operations.
The size of the numbers constructed by the algorithm is also polynomially bounded.
This is because each iteration starts with a face of the polyhedron $\P$, which can be obtained from its description by setting some inequality constraints to equality.
\end{proof}

%\section{Full-dimensional convex quadratic sets}

%\note{Continue from here, remove subsections if not needed! Also remove names of propositions if it makes sense.}

%\section{Proof of \cref{th main opt}}
%\section{Complexity of Mixed Integer Convex Quadratic Programming}
\section{Complexity of Mixed Integer Convex Quadratic Programming}
\label{sec proof of main}

%The main objective of this section is to show, in \cref{th main opt}, that MICQP is fixed parameter tractable (FPT) with parameter the number of integer variables.

In this section, we showcase the applicability and the potential impact of the fundamental results obtained in \cref{sec linear,sec cqs}.
%The main objective is to show that these results allow us 
We explain how these results can be used 
to revive Lenstra's original approach for ellipsoid rounding presented in \cite{Len83}.
%In turn, this allows us to design an algorithm for mixed integer convex quadratic programming (MICQP) that does not use the ellipsoid method as a subroutine, and that proves that mixed integer convex quadratic programming (MICQP) is fixed parameter tractable (FPT) with parameter the number of integer variables.
In turn, this allows us to design an algorithm for mixed integer convex quadratic programming that is FPT with parameter the number of integer variables.
This approach has the key advantage that it does not use the ellipsoid method as a subroutine.

A \emph{mixed integer convex quadratic programming} (MICQP) problem is defined as an optimization problem of the form
\begin{align}
\label[problem]{prob MICQP}
\tag{MICQP}
\begin{split}
\min & \quad x^\transp H x + h^\transp x \\
\st & \quad Wx \le w \\
& \quad x \in \Z^p \times \R^{n-p}.
\end{split}
\end{align}
Here $H \in PSD^n(\Q)$, which is the set of symmetric positive semidefinite matrices in $\Q^{n \times n}$, $h \in \Q^n$, $W \in \Q^{m \times n}$, $w \in \Q^m$, and $p \in \{0,1,\dots,n\}$.
%\footnote{The assumption that data in \cref{prob MICQP} is integer is without loss of generality, in the sense that rational data can be scaled to be integer 
%%in polynomial time
%at the expense of multiplying its size by $n^2$
%(see, e.g., remark~1.1 in \cite{ConCorZamBook}).}
%\footnote{It is well-know that we can verify in polynomial time whether $H \in PSD^n(\Z)$ by checking whether the leading principal minors of $H$ are nonnegative.}
%
%In this paper we construct an algorithm for the accurate solution of \cref{prob MICQP}.
Following \cite{KozTarKha81}, we say that an algorithm \emph{accurately solves} \cref{prob MICQP}
%of \cref{prob MICQP} 
%we mean the following:
if:
\begin{enumerate}[leftmargin=5mm]
\item
\textbf{(Feasibility)}
The algorithm determines whether \cref{prob MICQP} is \emph{feasible or infeasible,} i.e., if the \emph{feasible region} $\{x \in \Z^p \times \R^{n-p} : Wx \le w\}$
is empty or nonempty;
\item 
\textbf{(Boundedness)}
In case \cref{prob MICQP} is feasible, the algorithm establishes whether the problem is \emph{bounded or unbounded,} i.e., if the objective function is bounded or unbounded on the feasible region;
\item
\textbf{(Optimality)}
If \cref{prob MICQP} is feasible and bounded, the algorithm finds its minimal value and an \emph{optimal solution,} i.e., a point in the feasible region where the minimum is attained.
\end{enumerate}
%It is fundamental to observe that, as suggested by the wording ``accurate solution,'' our goal is significantly more demanding than finding an approximate solution, which is often the objective when considering convex optimization problems.
It is fundamental to observe that designing an algorithm that accurately solves \cref{prob MICQP} can be significantly more demanding than designing an algorithm that \emph{approximately} solves \cref{prob MICQP}, which is often the objective when considering convex optimization problems.
We are now ready to state our complexity result for MICQP.
%We are now ready to state our main result.
%\begin{restatable}[Optimization algorithm]{theorem}{mainoptimization}
%\begin{restatable}{theorem}{mainoptimization}
\begin{theorem}
\label{th main opt}
There is an algorithm that accurately solves \cref{prob MICQP}, 
which is FPT with parameter $p$.
\end{theorem}
%\end{restatable}

%In this section we leverage on the results obtained in \cref{sec linear,sec cqs} to prove the following result:

\subsection{L\"owner-John-type algorithm}
\label{sec cqs LJ}

%\begin{lemma}
%\label{lem sqrt approx}
%Let $p \in \Z$ be positive.
%There is a polynomial time algorithm which finds $\ceil{\sqrt{p}}$.
%\end{lemma}
%
%\begin{proof}
%Let $l_0:=0$, $u_0:=p$.
%Our algorithm consists of $k$ iterations $i=1,\dots,k$.
%At iteration $i$, set $m_i := \ceil{(l_{i-1}+u_{i-1})/2}$.
%Compare $m_i^2$ with $p$.
%If $m_i^2 < p$, set $l_{i}:=m_i$ and $u_{i}:=u_{i-1}$.
%If $m_i^2 \ge p$, set $l_{i}:=l_{i-1}$ and $u_{i}:=m_i$.
%When iteration $k$ is completed, 
%return $l:=l_k$ and $u:=u_k$.
%
%To show correctness of the algorithm, note that $\ceil{\sqrt{p}} \in \{0, 1, \dots, p\}$.
%The algorithm performs binary search on this interval, comparing, at each iteration, the ceiling $m_i$ of the midpoint of the previous interval with $\ceil{\sqrt{p}}$, and updating the interval accordingly.
%In fact, we have 
%\begin{align*}
%m_i \ge \ceil{\sqrt{p}} \quad \Leftrightarrow \quad m_i \ge \sqrt{p}  \quad \Leftrightarrow \quad  m_i^2 \ge p.
%\end{align*}
%Hence, for every $i=0,1,\dots,k$, $\ceil{\sqrt{p}}$ is contained in the interval $[l_i,u_i]$.
%Furthermore, for every $i=0,1,\dots,k$, we have $u_i-l_i \le p / 2^i$, thus the algorithm terminates in at most $\log(p)$ iterations.
%\end{proof}

%To prove \cref{claim small simplex}, we use a standard technique in linear optimization.

The first ingredient to prove \cref{th main opt} is a L\"owner-John-type algorithm for projected convex quadratic sets, which is of independent interest.
Given a convex quadratic set $\sQ$ in $\R^n$ that is bounded and full-dimensional, this algorithm constructs two concentric ellipsoids in $\R^p$ that sandwich the projection of $\sQ$ onto $\R^p$.
% and whose ratio is a function of $p$.
The ratio between the ellipsoids depends only on the dimension $p$ of the subspace rather than the dimension $n$ of the space.
Our proof is based on the technique introduced by Lenstra for mixed integer linear programming \cite{Len83} and makes use of the structural and algorithmic results obtained in \cref{sec cqs}.
%, which is often used to show results of this type.
%, and does not require the shallow cut ellipsoid method \cite{GroLovSch88}.
The key difficulty in extending Lenstra's technique lies in the fact that we are not able to minimize a linear function over $\sQ$ in polynomial time.
This is not surprising, given that this problem can have only irrational optimal solutions.
We overcome this obstacle in two different ways, depending on our goals: 
i) in \cref{claim small simplex} we employ \cref{prop full-dim,lem full-dim} to find a polytope $\F$ contained in $\sQ$ and then minimize over $\F$ instead of over $\sQ$;
% \note{explain full dim projection would be enough but simpler in full space?};
ii) in \cref{claim large simplex} we solve, instead, only feasibility problems over convex quadratic sets, which can be done with the Kozlov-Tarasov-Khachiyan algorithm \cite{KozTarKha81}.

We now introduce some notation which is needed to state this result.
A \emph{ball} in $\R^n$ with center $a \in \R^n$ and positive radius $r \in \R$ is a set of the form
\begin{align*}
\B^n(a,r) := \{x \in \R^n : \norm{x-a} \le r\}.
\end{align*}
Given a vector $v \in \R^n$, we denote by $\proj_p (v)$ the subvector of $v$ containing only the first $p$ components.
Given a set $\S \subseteq \R^n$, we denote by $\proj_p (\S)$ the orthogonal projection of $\S$ onto the space $\R^p$ of the first $p$ variables, i.e., $\proj_p (\S) = \{\proj_p (v) : v \in \S\}$.
%
%In the following, we denote by $I_{n}$ the $n \times n$ identity matrix.
%\note{Give overview of the proof.}
%An \emph{isomorphism} is a linear map $A \to B$ that is a bijection.
%An isomorphism with $A=B$ is called an \emph{automorphism} of $A$.
We also note that we will be using the number $\ceil{\sqrt{p}}$, for a positive integer $p$.
It is easy to see that such number can be computed, using binary search, in time polynomial in $\log (p)$.
We are now ready to present our L\"owner-John-type result.

\begin{proposition}
%[L\"owner-John-type algorithm]
\label{prop sandwich}
%Let $\sQ = \{x \in \R^n : Wx \le w, \  x^\transp H x + h^\transp x \le \eta \}$ be bounded and full-dimensional, where $W \in \Z^{m \times n}$, $w \in \Z^m$, $H \in PSD^n(\Z)$, $h \in \Z^n$, and $\eta \in \Z$.
Let $\sQ \subseteq \R^n$ be a convex quadratic set that is bounded and full-dimensional, and
let $p \in \{1,\dots,n\}$.
There is a polynomial time algorithm which finds 
%a bijective rational 
a map $\tau : \R^p \to \R^p$
of the form $\tau(y) = By$ with $B$ invertible in $\Q^{p \times p}$,
a vector $a \in \Q^p$, and positive numbers $r, R \in \Q$ satisfying $R / r \le 4 \ceil{\sqrt{p}}^3$ such that
\begin{align*}
\B^p(a, r) \subseteq \tau (\proj_p(\sQ)) \subseteq \B^p(a, R).
\end{align*}
\end{proposition}

\begin{proof}
Let $\sQ = \{x \in \R^n : Wx \le w , \ x^\transp H x + h^\transp x \le \eta \}$,  where $W \in \Q^{m \times n}$, $w \in \Q^m$, $H \in PSD^n(\Q)$, $h \in \Q^n$, and $\eta \in \Q$.
%, and
%let $\P := \{x \in \R^n : Wx \le w\}$.
%\note{Check if all used!}

\begin{claim}
\label{claim small simplex}
There is a polynomial time algorithm which finds affinely independent vectors $\tilde v^0, \tilde v^1, \ldots, \tilde v^p$ in $\proj_p(\sQ)$.
\end{claim}

\begin{cpf}
Since $\sQ$ is full-dimensional, using \cref{prop full-dim,lem full-dim}, we find a full-dimensional polytope $\F$ contained in $\sQ$.
To prove the claim, it suffices to find vectors $v^0, v^1, \ldots, v^p$ in $\F$ such that $\tilde v^0 := \proj_p (v^0), \tilde v^1 :=\proj_p (v^1), \ldots, \tilde v^p :=\proj_p (v^p)$ are affinely independent.
This can be done as follows.

The first vector $v^0$ in $\F$ can be found by minimizing an arbitrary linear function over $\F$, with Khachiyan's algorithm \cite{Kha79}.
%The vector $\tilde v^0 := \proj_p (v^0)$ is then in $\proj_p(\sQ)$.
Suppose now that we have found vectors $v^0, v^1, \ldots, v^t$ in $\F$, with $t<p$ such that $\tilde v^0, \tilde v^1, \ldots, \tilde v^t$ are affinely independent.
We construct $\tilde c \in \Q^p$ orthogonal to the affine hull of $\tilde v^0, \tilde v^1, \ldots, \tilde v^t$ and we define $c \in \Q^n$ as $c_j := \tilde c_j$ for $j=1,2,\dots,p$, and $c_j := 0$ for $j=p+1,p+2,\dots,n$.
We solve the two linear programming problems in $\R^n$ given by
\begin{align*}
\min\{c^\transp x : x \in \F \}, \qquad \max\{c^\transp x : x \in \F \}.
\end{align*}
Since $\F$ is bounded and full-dimensional, one of the two optimal solutions found must be a vector $v^{t+1}$ in $\F$ for which $c^\transp v^{t+1} \neq c^\transp v^0$.
Then $\tilde v^0, \tilde v^1, \ldots, \tilde v^t, \tilde v^{t+1}$ are affinely independent.
After $p$ steps we have found vectors $v^0, v^1, \ldots, v^p$ in $\F$ such that $\tilde v^0, \tilde v^1, \ldots, \tilde v^p$ are affinely independent.
\end{cpf}

\begin{claim}
\label{claim large simplex}
There is a polynomial time algorithm which finds vectors $\tilde v^0, \tilde v^1, \ldots, \tilde v^p$ in $\proj_p(\sQ)$ whose convex hull is given by inequalities ${\tilde c^i}^\transp y \le d_i$, $i=0,1,\dots,p$, such that, for all $i, j \in\{0,1, \ldots, p\}$ with $i \neq j$, we have 
\begin{align*}
%\label{eq ineq thingy}
\begin{split}
& {\tilde c^i}^\transp \tilde v^i < {\tilde c^i}^\transp \tilde v^j = d_i \\
& \abs{d_i - {\tilde c^i}^\transp y} \le \frac{3}{2} (d_i - {\tilde c^i}^\transp \tilde v^i) \qquad \forall y \in \proj_p(\sQ).
\end{split}
\end{align*}
\end{claim}

\begin{cpf}
Denote by $\S_0$ the simplex in $\R^p$ with vertices $\tilde v^0, \tilde v^1, \ldots, \tilde v^p$ obtained in \cref{claim small simplex}.
Clearly, $\S_0$ is contained in $\proj_p(\sQ)$.
Starting from $\S_0$, we iteratively apply the following procedure to generate a larger simplex in $\proj_p(\sQ)$.

\smallskip
\noindent
\ul{Iteration $t$ of the procedure.} 
Let $\S_{t-1}$ be the simplex in $\R^p$ with vertices $\tilde v^0, \tilde v^1, \ldots, \tilde v^p$.
We construct an inequality description ${\tilde c^i}^\transp y \le d_i$, $i=0,1,\dots,p$, of $\S_{t-1}$.
Without loss of generality we assume that, for every $i, j \in\{0,1, \ldots, p\}$ with $i \neq j$, we have ${\tilde c^i}^\transp \tilde v^i < {\tilde c^i}^\transp \tilde v^j = d_i$.
For each $i=0,1,\dots,p$, we define $c^i \in \Q^n$ as $c^i_j := \tilde c^i_j$ for $j=1,2,\dots,p$, and $c^i_j := 0$ for $j=p+1,p+2,\dots,n$, and we solve the two feasibility problems in $\R^n$ over the two sets
\begin{align}
\label{eq enlarge}
\bra{x \in \sQ : {c^i}^\transp x \le d_i - \frac 32 (d_i - {\tilde c^i}^\transp \tilde v^i) }, \quad 
\bra{x \in \sQ : {c^i}^\transp x \ge d_i + \frac 32 (d_i - {\tilde c^i}^\transp \tilde v^i) }.
\end{align}
Note that each such feasibility problem is over a convex quadratic set, and therefore it can be solved using the Kozlov-Tarasov-Khachiyan algorithm, by minimizing the convex quadratic function over the linear constraints.

If none of the problems {\eqref{eq enlarge}} is feasible, the procedure terminates.
Otherwise, at least one of the problems {\eqref{eq enlarge}} is feasible, we let ${v^i}'$ be a feasible point found, and we set ${\tilde v^i}' := \proj_p({v^i}')$.
Then, we define $\S_{t}$ to be the simplex in $\R^p$ obtained from $\S_{t-1}$ by replacing $\tilde v^i$ with ${\tilde v^i}'$, i.e., the simplex with vertices $\tilde v^0, \tilde v^1, \ldots, \tilde v^{i-1}, {\tilde v^i}', \tilde v^{i+1}, \ldots, \tilde v^p$.
Clearly, $\S_{t}$ is contained in $\proj_p(\sQ)$.
This concludes the description of iteration $t$ of the procedure.

\smallskip
\noindent
\ul{Termination of the procedure.} 
The above procedure terminates after a polynomially bounded number of iterations.
This is because of the following three facts:

%\noindent
Fact 1: 
$\vol(\S_0)$ is positive and its size is polynomial in $\size(H,h,W,w,\eta)$.
To show this fact, denote by $\tilde v^0, \tilde v^1, \ldots, \tilde v^p$ the vertices of $\S_0$ obtained in \cref{claim small simplex}.
The size of each vector $\tilde v^0, \tilde v^1, \ldots, \tilde v^p$ is polynomial in $\size(H,h,W,w,\eta)$, and $\delta = \abs{\det M} / p!$ where $M$ is the matrix with columns $\tilde v^1- \tilde v^0, \ldots, \tilde v^{p} - \tilde v^0$.

%\noindent
Fact 2: 
For each positive integer $t$, we have 
\begin{align*}
\vol(\S_{t}) \ge (3/2)^t \vol(\S_{0}).
\end{align*}
We now show this fact. 
From the definition of the sets \eqref{eq enlarge}, we obtain the following lower bound on the volume increase from $\S_{t-1}$ to $\S_{t}$:
\begin{align*}
\frac{\vol(\S_{t})}{\vol(\S_{t-1})} 
\ge 
\frac{\abs{d_i - {\tilde c^i}^\transp {\tilde v^i}'}}{d_i - {\tilde c^i}^\transp \tilde v^i}
=
\frac{\abs{d_i - {c^i}^\transp {v^i}'}}{d_i - {\tilde c^i}^\transp \tilde v^i}
\ge \frac 32.
\end{align*}
Hence,
\begin{align*}
\vol(\S_{t}) \ge 3/2 \vol(\S_{t-1}) \ge (3/2)^t \vol(\S_{0}).
\end{align*}

%\noindent
Fact 3: 
There exists $\Delta$, whose size is polynomial in $\size(H,h,W,w,\eta)$, such that, for every positive integer $t$, we have $\vol(\S_t) \le \Delta$.
We now show this fact.
For every $i=1,\dots,p$, consider the optimization problem
\begin{align*}
\min & \quad x_i \\
\st & \quad Wx \le w \\  
& \quad x^\transp H x + h^\transp x \le \eta.
\end{align*}
By scaling, we can assume without loss of generality that the data in this optimization problem is integer, and we denote by $s$ the size of the obtained problem.
% of the obtained problem with integer data.
%This can be done by scaling, at the expense of multiplying the size of the problem
%by $n^2$ (see, e.g., remark~1.1 in \cite{ConCorZamBook}).
%Denote by $s$ the size of the obtained problem with integer data.
It follows from theorem~2 in \cite{Kha83} that this problem has an optimal solution $x^*$ that satisfies
\begin{align*}
\norm{x^*} 
\le (2n 2^s)^{2^8 n}
\le (2s 2^s)^{2^8 s}
\le (2^{2s})^{2^8 s}
= 2^{2^9s^2}.
\end{align*}
%where we used $n \le s$ and $2s \le 2^s$.
Therefore, the set $\proj_p(\sQ)$ satisfies the inequalities $-2^{2^9s^2} \le y_i \le 2^{2^9s^2}$ for $i=1,2,\dots,p$.
% (see, for example, lemma~8.2 in \cite{BerTsi97}).
The fact then follows by observing that $\S_t \subseteq \proj_p(\sQ)$ for every positive integer $t$.

From Facts 1 to 3 we obtain, for each positive integer $t$,
\begin{align*}
\frac{\Delta}{\vol(\S_{0})} 
\ge 
\frac{\vol(\S_{t})}{\vol(\S_{0})} 
\ge
\pare{\frac 32}^t,
\end{align*}
thus the procedure terminates before iteration $t^* = \ceil{\log_{3/2}(\Delta/\vol(\S_{0}))}$, which is polynomially bounded.
From \eqref{eq enlarge}, at termination
we reach a situation in which,
for all $x \in \sQ$ and all $i = 0,1, \ldots, p$, we have  
\begin{align*}
\abs{d_i - {c^i}^\transp x} \le \frac{3}{2} (d_i - {\tilde c^i}^\transp \tilde v^i).
\end{align*}
Therefore, for all $y \in \proj_p(\sQ)$ and all $i = 0,1, \ldots, p$, we have  
\begin{align*}
\abs{d_i - {\tilde c^i}^\transp y} \le \frac{3}{2} (d_i - {\tilde c^i}^\transp \tilde v^i).
\end{align*}
\end{cpf}

Let $\tilde v^0, \tilde v^1, \ldots, \tilde v^p$ be the vectors in $\proj_p(\sQ)$ found by \cref{claim large simplex}.
%For each $i=0,1,\dots,p$, we define $v^i \in \Q^n$ as $v^i_j := \tilde v^i_j$ for $j=1,2,\dots,p$, and $v^i_j := 0$ for $j=p+1,p+2,\dots,n$.
We define the map $\tilde \tau : \R^p \to \R^p$ as 
\begin{align*}
\tilde \tau(y) := \tilde M^{-1} (y - \tilde v^0),
\end{align*} 
where $\tilde M \in \Q^{p \times p}$ is the invertible matrix with columns $\tilde v^1-\tilde v^0,\tilde v^2-\tilde v^0,\dots,\tilde v^p-\tilde v^0$.
Define
\begin{align*}
& \tilde a:=\frac{\tilde e}{p+\ceil{\sqrt{p}}} \in \Q^p \\
& r:=\frac{1}{p+\ceil{\sqrt{p}}} \in \Q \\
& R:= 2 \ceil{\sqrt{p}} \in \Q,
\end{align*}
where $\tilde e$ denotes the vector in $\R^p$ with all entries equal to one.
%From \cref{lem sqrt approx}, $\ceil{\sqrt{p}}$ can be computed in time polynomial in $\log (p)$.
Note that
\begin{align*}
\frac Rr = 2 \ceil{\sqrt{p}} (p+\ceil{\sqrt{p}}) \le 2 \ceil{\sqrt{p}}^3 + 2 \ceil{\sqrt{p}}^2 \le 4 \ceil{\sqrt{p}}^3.
\end{align*}

\begin{claim}
\label{eq claim 5}
We have
\begin{align*}
\B^p(\tilde a, r) \subseteq \tilde \tau (\proj_p(\sQ)) \subseteq \B^p(\tilde a,R).
\end{align*}
\end{claim}

\begin{cpf}
Note that $\tilde \tau(\tilde v^0) = 0$ and that $\tilde \tau(\tilde v^i) = \tilde e^i$, for $i=1,2,\dots,p$,
%, and $\tau(e^i) = e^i$, for \note{not sure} $i=p+1,p+2,\dots,n$, 
where we denote by $\tilde e^1,\tilde e^2,\dots,\tilde e^p$ the standard basis of $\R^p$.

We now show $\B^p(\tilde a, r) \subseteq \tilde \tau (\proj_p(\sQ)).$
To prove it, we denote by $\S$ the simplex in $\R^p$ defined by $\S := \conv\bra{0, \tilde e^1, \tilde e^2, \ldots, \tilde e^p}$.
An inequality description of $\S$ is given by $\S = \{z \in \R^p : z \ge 0, \ \sum_{i=1}^p z_i \le 1 \}$.
We then have $\B^p(\tilde a,r) \subseteq \S$, because
\begin{align*}
& \dist \pare{\tilde a, \ \tilde a- \frac{\tilde e^i}{p+\ceil{\sqrt{p}}}} = \frac{1}{p+\ceil{\sqrt{p}}} = r \qquad \forall i=1,\dots,p \\
& \dist\pare{\tilde a, \ \frac{\tilde e}{p}} 
= \sqrt{p} \frac{\ceil{\sqrt{p}}}{p(p+\ceil{\sqrt{p}})}
\ge \frac{1}{p+\ceil{\sqrt{p}}} = r.
\end{align*}
Since $\S \subseteq \tilde \tau(\proj_p(\sQ))$, we obtain $\B^p(\tilde a,r) \subseteq \tilde \tau(\proj_p(\sQ))$.

Next, we show $\tilde\tau (\proj_p(\sQ)) \subseteq \B^p(\tilde a, R)$.
From the conditions in \cref{claim large simplex}, we obtain 
\begin{align*}
\tilde\tau(\proj_p(\sQ)) 
& \subseteq \{z \in \R^p : -3/2 \le z \le 3/2, \ -1/2 \le \sum_{i=1}^p z_i \le 5/2 \} \\
& \subseteq \{z \in \R^p : -3/2 \le z \le 3/2 \}.
\end{align*}
We have
\begin{align*}
\dist \pare{\tilde a, -\frac32 \tilde e}
= \sqrt{p} \pare{ \frac 1 {p+\ceil{\sqrt{p}}} + \frac 32}
\le 2 \sqrt{p} \le R,
\end{align*}
thus $\tilde\tau(\proj_p(\sQ)) \subseteq \B^p(\tilde a,R)$.
\end{cpf}

Let $B := \tilde M^{-1} \in \Q^{p \times p}$, and define the map $\tau : \R^p \to \R^p$ as 
\begin{align*}
\tau(y) := By.
\end{align*}
Define
\begin{align*}
a:= \tilde a + B \tilde v_0 \in \Q^p.
\end{align*}

\begin{claim}
We have
\begin{align*}
\B^p(a, r) \subseteq \tau (\proj_p(\sQ)) \subseteq \B^p(a, R).
\end{align*}
\end{claim}

\begin{cpf}
%Due to \cref{eq claim 5}, it suffices to show 
%\begin{align*}
%\tilde \tau (\proj_p(\sQ)) = \proj_p(\tau (\sQ)).
%\end{align*}
Using the definition of $\tau$, for $y \in \R^p$, we have
\begin{align*}
\tau(y)
& = B y
=
\tilde \tau (y)  + B \tilde v_0.
\end{align*}
In particular, $\tau(\proj_p (\sQ)) = \tilde \tau (\proj_p (\sQ))  + B \tilde v_0$.
Thus, from \cref{eq claim 5}, we obtain
\begin{align*}
\B^p(a, r)
=
\B^p(\tilde a, r) + B \tilde v_0 
\subseteq 
\tau (\proj_p(\sQ)) 
\subseteq 
\B^p(\tilde a, R) + B \tilde v_0
=
\B^p(a, R).
\end{align*}
%If we denote by $e^1,e^2,\dots,e^p$ the standard basis of $\R^p$ and by $b^1,b^2,\dots,b^p \in \Q^p$ the columns of $B$, we have 
%\begin{align*}
%\tau(e^j)
%=
%\begin{pmatrix}
%b^j \\
%0
%\end{pmatrix}
%& \qquad j=1,2,\dots,p \\
%\tau(e^j) = e^j & \qquad j=p+1,p+2,\dots,n.
%\end{align*}
%Hence, 
%\begin{align*}
%\tau\pare{\Z^p \times \R^{n-p}} = \Lambda(B) \times \R^{n-p}.
%\end{align*}
\end{cpf}
\end{proof}

We remark that it might be possible to obtain an alternative proof of \cref{prop sandwich} using the shallow-cut ellipsoid algorithm.
For example, one could try to obtain a polynomial time weak separation oracle for projections of convex quadratic sets, and employ theorem~4.6.1 in~\cite{GroLovSch88}.
It is important to note that, in order to apply theorem~4.6.1 in~\cite{GroLovSch88}, the projection of the convex quadratic set must be well-bounded.
An advantage of our ``direct'' proof of \cref{prop sandwich} is that it does not use the ellipsoid method and does not require a weak separation oracle, or a well-bounded assumption.

\subsection{The feasibility problem}
\label{sec MICQP feas prob}

%\subsection{The feasibility algorithm}
%\label{sec MICQP feas alg}

In this section we show that the feasibility version of \cref{prob MICQP} is FPT with parameter $p$.
%We are now ready to give our feasibility result.
%The goal of this section is to prove the following theorem.
%The overall structure of the algorithm resembles most Lenstra-type algorithms in integer linear programming and in integer quasiconvex polynomial optimization.
The overall structure of the algorithm is based on Lenstra's \cite{Len83} algorithm for mixed integer linear programming.
The key difference is that in each iteration we reduce ourselves to the case where the convex quadratic set is full-dimensional.
This is in contrast to Lenstra's algorithm \cite{Len83}, where the linearity of the functions allows for a simple reduction to the case where the projection of the polyhedron onto the space of integer variables is full-dimensional.

Our proof uses 
\cref{th full-dim conv quad,prop sandwich,lem Lenstra,th full dim poly}, and a 
%\note{fix here: we added the lemma}
%
%In this section we show that the feasibility version of \cref{prob MICQP} is FPT with parameter $p$.
%This is a fundamental step in our proof of  first main step in proving \cref{th main opt} is to
%We now state our feasibility result.
%\subsection{Three standard lemmas}
%
%We start by presenting a % lemmas that will be used in our algorithm.
%
%The first one is a 
flatness result due to Lenstra \cite{Len83}, which we state below and that 
%We obtain it from remark~9.9 in \cite{ConCorZamBook} and, using \cref{lem direction transformation}, we restate it in terms of balls and a lattice, rather than in terms of ellipsoids and integer points.
%and that 
follows directly from proposition~4 in~\cite{dP23bMPA}.
Given an invertible matrix $B \in \R^{p \times p}$, we define the \emph{lattice} 
\begin{align*}
\Lambda(B)
 := \bra{B \mu : \mu \in \Z^p}.
\end{align*}
Recall that the \emph{width} of a bounded closed set $\S \subseteq \R^p$ along a vector $d \in \R^p$ is
\begin{align*}
\width_d (\S) := \max \bra{ d^\transp y : y \in \S } - \min \bra{ d^\transp y : y \in \S }.
\end{align*}

\begin{lemma}[Flatness lemma]
\label{lem Lenstra}
Let $a \in \Q^p$, let $r \in \Q$ with $r \ge 0$, and let $B \in \Q^{p \times p}$ be invertible.
There is a polynomial time algorithm which either finds a vector in $\B^p(a,r) \cap \Lambda(B)$,
or finds a vector $d \in \Q^p \setminus \{0\}$ with $B^\transp d \in \Z^p$
such that $\width_d(\B^p(a,r)) \le p \constlen$.
\end{lemma}

\begin{proposition} %[Feasibility algorithm]
\label{prop main feasibility}
Let $\sQ$ be a convex quadratic set in $\R^n$ and
%Let $\sQ = \{x \in \R^n : Wx \le w , \ x^\transp H x + h^\transp x \le \eta \}$,  where $W \in \Z^{m \times n}$, $w \in \Z^m$, $H \in PSD^n(\Z)$, $h \in \Z^n$, and $\eta \in \Z$.
let $p \in \{0,\dots,n\}$.
There is an algorithm that either returns a vector in $\sQ \cap \pare{\Z^p \times \R^{n-p}}$, or certifies that $\sQ \cap \pare{\Z^p \times \R^{n-p}} = \emptyset$.
%If it is nonempty, it returns a vector in $\sQ \cap \pare{\Z^p \times \R^{n-p}}$.
%The running time of the algorithm is polynomial in the size of $W,w,H,h,\eta$.
The algorithm is FPT with parameter $p$.
%is polynomial time, provided that the number $p$ of integer variables is fixed.
\end{proposition}

\begin{proof}
Let $\P := \{x \in \R^n : Wx \le w\}$, $\sQ = \{x \in \P : x^\transp H x + h^\transp x \le \eta \}$,  where $W \in \Q^{m \times n}$, $w \in \Q^m$, $H \in PSD^n(\Q)$, $h \in \Q^n$, and $\eta \in \Q$.
Let $\S := \sQ \cap \pare{\Z^p \times \R^{n-p}}.$
%, and let $\P := \{x \in \R^n : Wx \le w\}$.
%\note{Check if all used!}

%\textbf{Integer data.}
%From remark~1.1 in \cite{ConCorZamBook}, we can assume, without loss of generality, that $H,h,W,w,\eta$ have integer entries.
%The size of the resulting system of constraints is at most the square of the size of the original system.
%
%\textbf{Feasibility of the polyhedron.}
%First, we check if $\P$ is empty or not using Khachiyan's algorithm \cite{Kha79}.
%If $\P$ is empty, then so is $\sQ$, thus in the remainder of the proof we assume $\P$ nonempty.

\noindent
\textbf{Boundedness.}
By scaling, we can assume without loss of generality that the data defining $\sQ$ is integer, and we denote by $s$ the size of the obtained system.
%Denote by $s$ the size of the system defining $\sQ$.
%From theorem~4 in \cite{dPDeyMol17MPA}, there exists an optimal solution to \cref{prob MICQP}, which we denote by $x^*$.
It follows from theorem~1 in \cite{Kha83} that, if $\sQ \cap \pare{\Z^p \times \R^{n-p}}$ is nonempty, there exists $x^* \in \sQ \cap \pare{\Z^p \times \R^{n-p}}$ that satisfies
\begin{align*}
\norm{x^*} 
\le (2n 2^s)^{2^4 n}
\le (2s 2^s)^{2^4 s}
\le (2^{2s})^{2^4 s}
= 2^{2^5s^2}.
\end{align*}
%where we used $n \le s$ and $2s \le 2^s$.
Therefore, $x^*$ satisfies the $2n$ inequalities $-2^{2^5 s^2} \le x_i \le 2^{2^5 s^2}$, for $i = 1,2,\dots,n$.
Without loss of generality, we now assume that the system $Wx \le w$ contains these $2n$ inequalities, thus $\sQ$ is bounded.

\noindent
\textbf{Full-dimensionality.} 
%We apply \cref{prop full-dim} \note{check with new different proposition} to check whether $\sQ$ is full-dimensional.
% or not.
We apply \cref{th full-dim conv quad}.
%If $\sQ$ is not full-dimensional, \cref{lem full-dim} \note{check with new different proposition} returns a hyperplane $\H$ containing $\sQ$.
%We now apply \cref{lem go to lower} \note{this lemma will go} to $\sQ$ and $\H$.
If we find out that $\S$ is empty, we are done.
Otherwise, \cref{th full-dim conv quad} finds $p' \in \bra{0,1,\dots,p}$, $n' \in \bra{p',p'+1,\dots,p'+n-p}$, and a full-dimensional convex quadratic set $\sQ' \subseteq \R^{n'}$
%= \{y \in \R^{n-1} : W'y \le w', \  y^\transp H' y + {h'}^\transp y \le \eta' \}$, where $W' \in \Q^{m \times n-1}$, $w' \in \Q^m$, $H' \in PSD^{n-1}(\Q)$, $h' \in \Q^{n-1}$, and $\eta' \in \Q$, 
such that $\sQ \cap \pare{\Z^p \times \R^{n-p}}$ is empty if and only if $\sQ' \cap (\Z^{p'} \times \R^{n'-p'})$ is empty.
Furthermore, since $\sQ$ is bounded, then so is $\sQ'$.
%Recursively applying this step, we either detect that $\sQ \cap (\Z^{p} \times \R^{n-p})$ is empty, or we reduce ourselves to an equivalent full-dimensional problem.
For ease of notation, we simply assume that $\sQ$ is full-dimensional.

\noindent
\textbf{Continuous case.}
If $p=0$, we can solve the convex quadratic programming problem
\begin{align*}
\min & \quad x^\transp H x + h^\transp x \\
\st & \quad Wx \le w 
\end{align*}
using Kozlov-Tarasov-Khachiyan algorithm \cite{KozTarKha81}.
We denote by $\bar \eta$ the minimal value and we have that $\sQ$ is empty if and only if $\eta < \bar \eta$.
Thus, in the remainder of the proof, we assume $p \ge 1$.

\noindent
\textbf{Feasibility or partition.}
We apply \cref{prop sandwich} and find 
a map $\tau : \R^p \to \R^p$ of the form $\tau(y) = By$ with $B$ invertible in $\Q^{p \times p}$,
a vector $a \in \Q^p$,
and positive numbers $r, R \in \Q$ satisfying $R / r \le 4 \ceil{\sqrt{p}}^3$ such that
\begin{align*}
\B^p(a, r) \subseteq \tau (\proj_p(\sQ)) \subseteq \B^p(a, R).
\end{align*}
From \cref{lem Lenstra}, there is an algorithm which either finds a vector $\tilde z$ in $\B^p(a,r) \cap \Lambda(B)$, 
or finds a vector ${\tilde d} \in \Q^p \setminus \{0\}$ with $B^\transp \tilde d \in \Z^p$ such that $\width_{\tilde d}(\B^p(a,r)) \le p \constlen$.

In the remainder of the proof, we consider separately two cases.
First, consider the case where \cref{lem Lenstra} found a vector $\tilde z$ in $\B^p(a,r) \cap \Lambda(B)$.
Then, denoting by $\tau^\leftarrow$ the inverse of $\tau$, the vector $\tilde y := \tau^\leftarrow(\tilde z) = B^{-1} \tilde z$ is in $\tau^\leftarrow(\B^p(a, r))$, and thus in $\proj_p(\sQ)$.
From $\tilde z \in \Lambda(B)$, we obtain $\tilde y \in \Z^p$.
Thus, $\tilde y \in \proj_p(\sQ) \cap \Z^p$.
This implies that $\sQ \cap \pare{\Z^p \times \R^{n-p}}$ is nonempty.
A vector in $\sQ \cap \pare{\Z^p \times \R^{n-p}}$ can be found 
by solving the following convex quadratic programming problem using Kozlov-Tarasov-Khachiyan algorithm:
\begin{align*}
\min & \quad x^\transp H x + h^\transp x \\
\st & \quad Wx \le w \\
& \quad x_i = \tilde y_i & i=1,2,\dots,p.
\end{align*}

Next, consider the case where \cref{lem Lenstra} found a vector ${\tilde d} \in \Q^p \setminus \{0\}$ with $B^\transp \tilde d \in \Z^p$ 
%and ${\tilde d}^\transp b^{p+1} = \cdots = {\tilde d}^\transp b^n = 0$ 
such that $\width_{\tilde d}(\B^p(a,r)) \le p \constlen$. 
Hence,
\begin{align*}
\width_{\tilde d}(\tau(\proj_p(\sQ)))
\le \width_{\tilde d}(\B^p(a,R)) 
= \frac Rr \width_{\tilde d}(\B^p(a,r)) 
\le 4 \ceil{\sqrt{p}}^3 p \constlen.
\end{align*}
Let $\tilde c := \basmatB^\transp {\tilde d} \in \Z^p \setminus \{0\}$.
%Using \cref{lem direction transformation}, 
Then,
\begin{align*}
\width_{\tilde c}(\proj_p(\sQ))
& = \max \bra{ {\tilde c}^\transp y : y \in \proj_p(\sQ) } - \min \bra{ {\tilde c}^\transp y : y \in \proj_p(\sQ) } \\
& = \max \bra{ {\tilde d}^\transp B y : y \in \proj_p(\sQ) } - \min \bra{ {\tilde d}^\transp B y : y \in \proj_p(\sQ) } \\
& = \max \bra{ {\tilde d}^\transp z : z \in \tau(\proj_p(\sQ)) } - \min \bra{ {\tilde d}^\transp z : z \in \tau(\proj_p(\sQ)) } \\
& = \width_{\tilde d}(\tau(\proj_p(\sQ))) \\
%\end{align*}
%
%$d_{p+1} = \cdots = d_n = 0$, and 
%if we denote by $\tau^\leftarrow$ the inverse of the affine map $\tau$, 
%\begin{align*}
%\width_{\tilde c}(\proj_p(\sQ))
%= \width_{\tilde d}(\tau(\proj_p(\sQ)))
& \le 4 \ceil{\sqrt{p}}^3 p \constlen.
\end{align*}
Let 
\begin{align*}
\rho & := \min\{{\tilde c}^\transp y : y \in \tau^\leftarrow(\B^p(a,R))\} \\
& = \min\{{\tilde d}^\transp z : z \in \B^p(a,R)\} \\
& = \min\{{\tilde d}^\transp z : z \in \B^p(0,R)\} + {\tilde d}^\transp a \\
& = - R \norm{\tilde d} + {\tilde d}^\transp a.
\end{align*}
We obtain
\begin{align*}
\{{\tilde c}^\transp y : y \in \tau^\leftarrow(\B^p(a,R))\} \in [\rho, \rho + 4 \ceil{\sqrt{p}}^3 p \constlen].
\end{align*}
Since $\proj_p(\sQ) \subseteq \tau^\leftarrow(\B^p(a,R))$, we also have
\begin{align*}
\{{\tilde c}^\transp y : y \in \proj_p(\sQ)\} \in [\rho, \rho + 4 \ceil{\sqrt{p}}^3 p \constlen].
\end{align*}
Every point in $\proj_p(\sQ) \cap \Z^p$ is contained in one of the hyperplanes
\begin{align*}
\{y \in \R^p : {\tilde c}^\transp y = \gamma\}, \qquad \gamma = \ceil{\rho}, \ \ceil{\rho} + 1, \dots, \ceil{\rho + 4 \ceil{\sqrt{p}}^3 p \constlen}.
\end{align*}
We define $c \in \Z^n \setminus \{0\}$ as $c_j := \tilde c_j$ for $j=1,2,\dots,p$, and $c_j := 0$ for $j=p+1,p+2,\dots,n$.
Then, every point in $\sQ \cap \pare{\Z^p \times \R^{n-p}}$ is contained in one of the hyperplanes
\begin{align*}
\H_\gamma := \{x \in \R^n : {c}^\transp x = \gamma\}, \qquad \gamma = \ceil{\rho}, \ \ceil{\rho} + 1, \dots, \ceil{\rho + 4 \ceil{\sqrt{p}}^3 p \constlen}.
\end{align*}
For each $\gamma = \ceil{\rho}, \ceil{\rho} + 1, \dots, \ceil{\rho + 4 \ceil{\sqrt{p}}^3 p \constlen}$, 
%we define the set
%\begin{align*}
%\{x \in \R^n : c^\transp x = \gamma \},
%\end{align*}
we apply \cref{th full dim poly} to the polyhedron $\P \cap \H_\gamma$.
% and the quadratic inequality $x^\transp H x + h^\transp x \le \eta \}$.
If we find out that $\P \cap \H_\gamma \cap \pare{\Z^p \times \R^{n-p}}$ is empty, then there is no need to consider this value $\gamma$ any further.
Otherwise, \cref{th full dim poly} finds $p' \in \bra{0,1,\dots,p-1}$, $n' \in \bra{p',p'+1,\dots,p'+n-p}$, 
a map $\tau : \R^{n'} \to \R^n$
of the form $\tau(x') = \bar x + Mx'$, with 
$\bar x \in \Z^p \times \Q^{n-p}$ and $M \in \Q^{n \times n'}$ of full rank,
and 
a full-dimensional polyhedron 
$\P'_\gamma \subseteq \R^{n'}$
such that
\begin{align*}
\P \cap \H_\gamma & = \tau \pare{\P'_\gamma} \\
\P \cap \H_\gamma \cap \pare{\Z^p \times \R^{n-p}} & = \tau \pare{\P'_\gamma \cap \pare{\Z^{p'} \times \R^{n'-p'}}}.
\end{align*}
We define %$W' := WM \in \Q^{m \times n'}$, $w' := w - W \bar x \in \Q^m$,
$H' := M^\transp H M \in PSD^{n'}(\Q)$,
$h' := 2M^\transp H^\transp \bar x + M^\transp h \in \Q^{n'}$, 
$\eta' := \eta - \bar x^\transp H \bar x + h^\transp \bar x \in \Q$,
and the convex quadratic set 
%\begin{align*}
%\sQ' := \bra{x' \in \R^{n'} : WMx' \le w - W \bar x, \ {x'}^\transp (M^\transp H M) x' + {(2M^\transp H^\transp \bar x + M^\transp h)}^\transp x' \le (\eta - \bar x^\transp H \bar x + h^\transp \bar x)}
%\end{align*} 
\begin{align*}
\sQ' := \bra{x' \in \P'_\gamma : {x'}^\transp H' x' + {h'}^\transp x' \le \eta'}.
\end{align*} 
We then have
%is full-dimensional, and 
\begin{align*}
%\bra{x \in \R^n : Wx \le w} & = \bra{\bar x + M x' : W'x' \le w'}, \\
%\sQ & = \tau\pare{x' \in \R^{n'} : \bra{{x'}^\transp H' x' + {h'}^\transp x' \le \eta'}}, \\
%\Z^p \times \R^{n-p} & = \tau(\Z^{p'} \times \R^{n'-p'}).
\sQ \cap \H_\gamma & = \tau \pare{\sQ'_\gamma} \\
\S \cap \H_\gamma = \sQ \cap \H_\gamma \cap \pare{\Z^p \times \R^{n-p}} & = \tau \pare{\sQ'_\gamma \cap \pare{\Z^{p'} \times \R^{n'-p'}}}.
\end{align*}

%and a polyhedron $\P_\gamma \subseteq \R^{n'}$ with 
% = \{y \in \R^{n-1} : W'y \le w' \}$, $\sQ_\gamma = \{y \in \P'_{\gamma} : y^\transp H' y + {h'}^\transp y \le \eta' \}$, where $W' \in \Q^{m \times n-1}$, $w' \in \Q^m$, $H' \in PSD^{n-1}(\Q)$, $h' \in \Q^{n-1}$, $\eta' \in \Q$, 
%and a map $\pi : \R^{n-1} \to \H_\gamma$
%of the form $\pi(y) = \bar x + My$ with $\bar x \in \H_\gamma$ and $M$ of rank $n-1$ in $\Q^{n\times n-1}$,
%such that 
%%$\pi(\P_\gamma) = \P \cap \H_\gamma$,
%$\pi(\sQ_\gamma) = \sQ \cap \H_\gamma$, and 
%$\pi(\Z^{p-1} \times \R^{n-p}) = \H_\gamma \cap \pare{\Z^p \times \R^{n-p}}$.
%Therefore, 
%\begin{align*}
%\pi(\sQ_\gamma \cap (\Z^{p-1} \times \R^{n-p})) = \sQ \cap \H_\gamma \cap \pare{\Z^p \times \R^{n-p}}.
%\end{align*}
We can then solve the feasibility problem over $\sQ \cap \H_\gamma \cap \pare{\Z^p \times \R^{n-p}}$ by solving, instead, the feasibility problem over $\sQ'_\gamma \cap (\Z^{p'} \times \R^{n'-p'})$.
As a result, we can solve the feasibility problem over $\sQ \cap \pare{\Z^p \times \R^{n-p}}$ by solving, instead, all the feasibility problems over $\sQ'_\gamma \cap (\Z^{p'} \times \R^{n'-p'})$, for $\gamma = \ceil{\rho}, \ceil{\rho}+1, \dots, \ceil{\rho + 4 \ceil{\sqrt{p}}^3 p \constlen}$.
%From remark~1.1 in \cite{ConCorZamBook}, we can assume, without loss of generality, that all data defining $\sQ_\gamma$ is integer.
%Note that, since $\sQ$ is bounded, then so is $\sQ_\gamma$, thus, for each $\gamma$, we go back to Step~2.
This concludes one iteration of the algorithm.

\noindent
\textbf{Recursion.}
We apply recursively the iteration of the algorithm described so far in the proof.
Note that, since $\sQ$ is bounded, then so is each $\sQ_\gamma$, thus Step~1 needs to be performed only in the very first iteration.
If, in some iteration, a feasible solution of a subproblem is found, we can find a vector in $\sQ \cap \pare{\Z^p \times \R^{n-p}}$ by inverting all the maps $\tau$ used, in the previous iterations, to obtain the subproblem. 
Note that the inverse of $\tau : \R^{n'} \to \H_\gamma$ given by $\tau(x') = \bar x + Mx'$ is $\tau^\leftarrow(x) = (M^\transp M)^{-1} M^\transp (x-\bar x)$, since $M$ has full column rank.

Since each time the number of integer variables decreases at least by one,
the total number of iterations is upper bounded by
\begin{align*}
O\pare{\pare{4 \ceil{\sqrt{p}}^3 p \constlen}^p} = O\pare{p^{3p/2} p^p 2^{p^2(p-1)/4}}.
\end{align*}
\end{proof}

\subsection{Boundedness of \cref{prob MICQP}}
\label{sec MICQP bounded}

%To design an algorithm that accurately solves \cref{prob MICQP}, a fundamental first step consists in establishing whether the problem is bounded or unbounded, i.e., if the objective function is bounded or unbounded on the feasible region.
In this section we characterize when \cref{prob MICQP} is bounded. 
This characterization allows us to obtain an FPT algorithm to check boundedness of the problem. This algorithm only solves one mixed integer linear feasibility problem and one linear feasibility problem.
We remark that the convexity of the objective function is essential to obtain this result.
%In fact, unless P=NP, \cref{prop unbounded} does not hold if we drop the convexity assumption in \cref{prob MICQP}.
In fact, determining whether a nonconvex quadratic programming problem is bounded 
%or not 
is NP-hard \cite{MurKab87}, even if the rank of the quadratic matrix $H$ is three \cite{dP23bMPA}.
In what follows, for a polyhedron $\P = \{x \in \R^n : Wx \le w\}$, we denote by $\rec(\P)$ its recession cone
\begin{align*}
\rec(\P) := \{r \in \R^n : x + r \in \P \ \forall x \in \P \} = \{r \in \R^n : Wx \le 0\}.
\end{align*}

\begin{proposition}
\label{prop unbounded}
Consider \cref{prob MICQP} and let $\P := \{x \in \R^n : Wx \le w\}$.
\cref{prob MICQP} is unbounded if and only if the two sets $\P \cap \pare{\Z^p \times \R^{n-p}}$ and $\{ r \in \R^n : Wr \le 0, \ Hr = 0, \ h^\transp r \le -1 \}$ are both nonempty.
%$\min \{ h^\transp r : Hr = 0, \ Wr \le 0\} < 0$.
Furthermore, there is an algorithm that detects whether \cref{prob MICQP} is bounded or unbounded, which is FPT with parameter $p$.
% or not.
If it is unbounded, it finds $\bar x \in \P \cap \pare{\Z^p \times \R^{n-p}}$ and $\bar r \in \rec(\P)$ such that the objective function goes to minus infinity on the half-line $\{\bar x + \lambda \bar r : \lambda \ge 0\}$.
\end{proposition}

\begin{proof}
We start by proving the ``if and only if'' in the statement.
% our characterization of unboundedness of \cref{prob MICQP}.
%\begin{claim}
%\label{cla unb}
%\cref{prob MICQP} is unbounded if and only if $\P \cap \pare{\Z^p \times \R^{n-p}}$ is nonempty and $\min \{ h^\transp r : Hr = 0, \ Wr \le 0\} < 0$.
%\end{claim}
%
%\begin{cpf}
From theorem 4 in~\cite{dPDeyMol17MPA}, \cref{prob MICQP} is unbounded if and only if 
%there exist $\bar x$ and $\bar r$ as in the statement.
there exist $\bar x \in \P \cap \pare{\Z^p \times \R^{n-p}}$ and $\bar r \in \rec(\P)$ such that the objective function $q(x) := x^\transp H x + h^\transp x$ goes to minus infinity on the half-line $\{\bar x + \lambda \bar r : \lambda \ge 0\}$.
Evaluating $q(x)$ on the half-line, we obtain
\begin{align*}
q(\bar x + \lambda \bar r) & 
= (\bar x + \lambda \bar r)^\transp H (\bar x + \lambda \bar r) + h^\transp (\bar x + \lambda \bar r) \\
& = \lambda^2(\bar r^\transp H \bar r) + \lambda (2 \bar x^\transp H \bar r + h^\transp \bar r) + (\bar x^\transp H \bar x + h^\transp \bar x).
\end{align*}
We observe that $q(x)$ goes to minus infinity on the half-line if and only if either $\bar r^\transp H \bar r < 0$, or $\bar r^\transp H \bar r = 0$ and $2 \bar x^\transp H \bar r + h^\transp \bar r < 0$.
Since $H$ is positive semidefinite, $\bar r^\transp H \bar r \ge 0$ and $\bar r^\transp H \bar r = 0$ if and only if $H \bar r = 0$.
Hence $q(x)$ is unbounded on the above half-line if and only if $H \bar r = 0$ and $h^\transp \bar r < 0$.
So far we have shown that \cref{prob MICQP} is unbounded if and only if the two sets $\P \cap \pare{\Z^p \times \R^{n-p}}$ and $\{ r \in \R^n : Wr \le 0, \ Hr = 0, \ h^\transp r < 0\}$ are both nonempty.
Note that the set $\{ r \in \R^n : Wr \le 0, \ Hr = 0\}$ is a cone, thus the second set is nonempty if and only if the set $\{ r \in \R^n : Wr \le 0, \ Hr = 0, \ h^\transp r \le -1\}$ is nonempty.
This completes the proof of the ``if and only if'' in the statement.

We solve the mixed integer linear feasibility problem over 
\begin{align}
\label{eq unb one}
\{x \in \Z^p \times \R^{n-p} : Wx \le w\}
\end{align}
with Lenstra's algorithm \cite{Len83} (or with the algorithm in \cref{prop main feasibility}) and the linear feasibility problem over
\begin{align}
\label{eq unb two}
\{ r \in \R^n : Wr \le 0, \ Hr = 0, \ h^\transp r \le -1\}
\end{align}
with Khachiyan's algorithm \cite{Kha79}.
% the Kozlov-Tarasov-Khachiyan algorithm \cite{KozTarKha81}.
If at least one of the two problems \eqref{eq unb one}, \eqref{eq unb two} is infeasible, then \cref{prob MICQP} is bounded.
Assume now that both problems \eqref{eq unb one} and \eqref{eq unb two} are feasible.
In this case \cref{prob MICQP} is unbounded.
Furthermore, the vector $\bar x$ in the statement is a vector in \eqref{eq unb one} that we found, and the vector $\bar r$ in the statement is a vector in \eqref{eq unb two} that we found.
%If the minimal value of \eqref{eq unb two} is finite, then the vector $\bar r$ is the optimal solution of \eqref{eq unb two} that we found.
%
%Note that the feasible region of \eqref{eq unb two} is a cone, thus the minimal value of \eqref{eq unb two} is $-\infty$.
%Then the vector $\bar r$ is a feasible solution to 
%\begin{align*}
%%\label{eq unb two}
%\{ r \in \R^n : h^\transp r \le -1, \ Hr = 0, \ Wr \le 0\}
%\end{align*}
%that we can find with Khachiyan's algorithm \cite{Kha79}.
%If the first problem is infeasible or the minimal value of the second problem is nonnegative, \cref{prob MICQP} is bounded.
%Otherwise, \cref{prob MICQP} is unbounded.
% and we now show how to find $\bar x, \bar r$ as in the statement.
%The vector $\bar x$ is simply the vector in $\{x \in \Z^p \times \R^{n-p} : Wx \le w\}$ that we found.
%If the problem $\min \{ h^\transp r : Hr = 0, \ Wr \le 0\}$ is bounded, the vector $\bar r$ is the optimal solution found.
%Otherwise, if the problem $\min \{ h^\transp r : Hr = 0, \ Wr \le 0\}$ is unbounded, we solve the convex quadratic feasibility problem over 
%the vector $\bar r$ is the optimal solution found.
%If the problem 
%, and an optimal solution
%If the optimum value is strictly negative, the optimal solutions provides the vectors $\bar x, \bar r$ as in the statement.
\end{proof}

%shown in theorem~2 in \cite{dP23bMPA} that , even if the 
% if the rank k of the matrix H equals three and the number p of integer variables is zero.

% the above is NP-hard in the non-convex case, even fully continuous. Also NP-hard in linear case?

%\subsection{Bounds on solutions}

%\begin{lemma}
%\label{lem size}
%Consider \cref{prob MICQP} and assume that it is feasible.
%Denote by $s$ the size of \cref{prob MICQP}.
%There is an optimal solution $x^*$ to \cref{prob MICQP} such that $\norm{x^*} \le 2^{2^7 s^2}$.
%Each entry of $x^*$ can be written as the ratio of two integer numbers with absolute value at most $2^{2^8 s^3}$.
%The minimal value of \cref{prob MICQP} can be written as $\alpha/\beta$, with $\alpha,\beta \in \Z$ such that $\abs{\alpha} \le 2^{2^{10} s^3}$ and $\abs{\beta} \le 2^{2^9 s^3}$.
%\end{lemma}
%
%\begin{proof}
%\end{proof}

%We remark that better bounds on the size of optimal solutions can be obtained using, for example, theorem 6.4 in \cite{Kop12}, which provides the bound $O(\ell) 2^{O(n)}$, where $\ell$ is an upper bound on the size of the coefficients. However, all the references we are aware of use the big O notation and do not give explicit constants.

\subsection{The optimization problem}
\label{sec MICQP opt prob}

We are now ready to prove \cref{th main opt} with the classic technique of binary search.
%We now get back to our overview of \cref{th main opt}.
%Our optimization algorithm performs the following three main steps:
%
%\note{fix overview below}
%
%\begin{itemize}
%\item[]
%\textbf{Step O1.}
%We detect whether the feasible region of \cref{prob MICQP} is empty or nonempty. This can be done using \cref{prop main feasibility} or Lenstra's algorithm \cite{Len83}.
%
%\item[]
%\textbf{Step O2.}
%We detect whether \cref{prob MICQP} is bounded.
%% or not.
%The result at the basis of this step is \cref{prop unbounded}, where we characterize when \cref{prob MICQP} is bounded and obtain an FPT algorithm to check it.
%
%\item[]
%\textbf{Step O3.}
%In case \cref{prob MICQP} is bounded, we find the minimal value and an optimal solution invoking \cref{prop main feasibility} a polynomial number of times in a binary search environment.
%%, using bounds from \cite{Kha83,KozTarKha81}.
%\end{itemize}
%
%An interesting question is whether it is possible to replace binary search with an oracle-based approach as in \cite{Dadush12,OerWagWei14}.

%algorithm (see, e.g., \cite{GroLovSch88}), or a Lenstra-type algorithm (see, e.g., \cite{Len83}).
%We choose the second option, and as a result the overall structure of our algorithm resembles most 

%This concludes our summary of the optimization algorithm.

%\subsection{The optimization algorithm}
%\label{sec MICQP opt alg}

%We are now ready to prove \cref{th main opt}.
%\mainoptimization*

\begin{prfc}[of \cref{th main opt}]
%\textbf{Integer data.}
%From remark~1.1 in \cite{ConCorZamBook}, we can assume, without loss of generality, that $H,h,W,w$ have integer entries.
%The size of the resulting problem is at most the square of the size of the original problem.
Let $\P := \{x \in \R^n : Wx \le w\}$.

\noindent
\textbf{Feasibility.} 
We use Lenstra's algorithm \cite{Len83} (or \cref{prop main feasibility}) to check whether the feasible region $\P \cap \pare{\Z^p \times \R^{n-p}}$ is empty or nonempty.
If it is empty we are done, thus we now assume that it is nonempty.

\noindent
\textbf{Boundedness.} 
We use \cref{prop unbounded} to detect whether \cref{prob MICQP} is bounded or unbounded.
If it is unbounded we are done, thus we now assume that it is bounded.

\noindent
%\textbf{Boundedness of the polyhedron.}
\textbf{Binary search.}
By scaling, we can assume without loss of generality that the data in \cref{prob MICQP} is integer, and we denote by $s$ its size.
%Denote by $s$ the size of \cref{prob MICQP}.
%From theorem~4 in~\cite{dPDeyMol17MPA}, there exists an optimal solution to \cref{prob MICQP}, which we denote by $x^*$.
It follows from theorem~2 in \cite{Kha83} that \cref{prob MICQP} has an optimal solution $x^*$ that satisfies
\begin{align*}
\norm{x^*} 
\le (2n 2^s)^{2^6 n}
\le (2s 2^s)^{2^6 s}
\le (2^{2s})^{2^6 s}
= 2^{2^7s^2}.
\end{align*}
%where we used $n \le s$ and $2s \le 2^s$.
Therefore, $x^*$ satisfies the $2n$ inequalities $-2^{2^7 s^2} \le x_i \le 2^{2^7 s^2}$, for $i = 1,2,\dots,n$.
Without loss of generality, we now assume that the system $Wx \le w$ contains these $2n$ inequalities.
%
%\textbf{Step 3. Binary search.}
Thus, the absolute value of the objective value of each vector $x \in \P$ can be bounded as follows:
%From \cref{lem size}, there is an optimal solution to \cref{prob MICQP} that satisfies $\norm{x} \le 2^{2^7 s^2}$, and so it satisfies the $2n$ inequalities $-2^{2^7 s^2} \le x_i \le 2^{2^7 s^2}$, for $i = 1,2,\dots,n$.
%Without loss of generality, we now assume that the system $Wx \le w$ contains these $2n$ inequalities. %, thus $\P$ is bounded.
%
%\textbf{Optimality.} 
%From \cref{prop unbounded}, we obtain a positive integer $\psi$, polynomial in the size of \cref{prob MICQP}, such that there is an optimal solution to \cref{prob MICQP} of size at most $\psi$.
%In particular, there is an optimal solution to \cref{prob MICQP} that satisfies the inequalities $-2^\psi \le x_i \le 2^\psi$, for $i = 1,2,\dots,n$.
%Without loss of generality, we now assume that the system $Wx \le w$ contains these $2n$ inequalities, thus $\P$ is bounded.
%
%Let $\nu$ be the largest size of the entries of $H,h$.
\begin{align*}
\absL{\sum_{i,j=1}^n H_{ij} x_i x_j + \sum_{i=1}^n h_i x_i}
& \le n^2 2^{s} (2^{2^7 s^2})^2 + n 2^{s} 2^{2^7 s^2}
\le s^2 2^{s+2^8 s^2} + s 2^{s +2^7 s^2} \\
& \le 2 s^2 2^{s+2^8 s^2}
% \le 2^{2 \log(s) + s+1+2^8 s^2}
\le 2^{2^9 s^2}.
\end{align*}

Next, we bound the absolute value of the denominator of the minimal value of \cref{prob MICQP}.
Denote again by $x^*$ an optimal solution to \cref{prob MICQP}.
Consider now the equivalent convex quadratic programming problem, with $n$ variables and integer data, obtained from \cref{prob MICQP} by adding the $p$ equality constraints $x_i = x^*_i$ for $i=1,\dots,p$.
The size of this continuous problem is at most 
\begin{align*}
s' : = s + 2np + p (2^7 s^2) \le 2^7 s^3.
\end{align*}
From assertion~3 in~\cite{KozTarKha81}, the minimal value of this continuous problem can be written as the ratio of two integer numbers, where the absolute value of the denominator is at most 
\begin{align*}
2^{4 s'} = 2^{2^9 s^3}.
\end{align*}
The same conclusion holds for the minimal value of \cref{prob MICQP}, since the two problems are equivalent.

%Furthermore, from \cref{lem size}, the minimal value of \cref{prob MICQP} can be written as the ratio of two integer numbers, where the absolute value of the denominator is at most 
%\begin{align*}
%2^{2^9 s^3}.
%\end{align*}
We combine the above two bounds and obtain that, to solve \cref{prob MICQP}, it suffices to solve
\begin{align*}
\log(2^{2^9 s^2+1+2^9 s^3})
\le \log(2^{2^{10} s^3})
= 2^{10} s^3
\end{align*}
mixed integer feasibility problems on a convex quadratic set of the form 
\begin{align*}
\sQ_\eta = \{x \in \R^n : Wx \le w, \ x^\transp H x + h^\transp x \le \eta \}.
\end{align*}
Each such feasibility problem can be solved using \cref{prop main feasibility}.
\end{prfc}

%\subsection{Alternative strategies to prove \cref{th main opt}}

%We conclude this paper by discussing some connections between \cref{th main opt} and known results in integer convex optimization.
%, alternative strategies to prove \cref{th main opt}.
We believe that the interest of \cref{th main opt} lies not only in its statement, but also in its proof.
In particular, the proof shows how the results obtained in \cref{sec linear,sec cqs} can be used 
to revive Lenstra's original approach for ellipsoid rounding.
As a result, we are able to design an FPT algorithms for MICQP that does not use the ellipsoid method as a subroutine.
%Due to the convexity Since MICQP 

If we focus only on the statement of \cref{th main opt},
%Due to the convexity of the objective function, 
a natural question is whether 
%That said, it 
it is possible to obtain an alternative proof using oracle-based techniques from integer convex optimization in fixed dimension.
%, like those used in \cite{Dadush12,OerWagWei14,VesGriZolChi20}.
% and $\S := \P \cap \pare{\Z^p \times \R^{n-p}}$.
%
In an attempt to apply these techniques, we can first reduce ourselves to the case where $\P := \bra{x \in \R^n : Wx \le w}$ is bounded, which can be done as discussed in the first part of the proof of \cref{th main opt}, using our \cref{prop unbounded} and theorem~2 in \cite{Kha83}.
Next, we cast \cref{prob MICQP} as a pure integer optimization problem with feasible region $\proj_p (\P) \cap \Z^p$ and objective function $f : \proj_p (\P) \to \Q$ defined by
\begin{align*}
f(x_1,\dots,x_p) := \min_{x_{p+1},\dots,x_n} \bra{x^\transp H x + h^\transp x : Wx \le w }.
\end{align*}
Note that the function $f$ is convex, and 
%In fact, given a vector $(\bar x_1,\dots, \bar x_p) \in \proj_p (\P)$, 
we can obtain a polynomial time evaluation oracle for $f$ using the Kozlov-Tarasov-Khachiyan algorithm.
We observe that we can also define an extension $f':\R^p \to \R$ of $f$ by setting $f'(x_1,\dots,x_p)$ equal to some large number $M$, in case $(x_1,\dots,x_p) \notin \proj_p (\P)$.
The obtained function $f'$ is quasiconvex, but it is not convex, conic, or discrete convic.

% to find $\bar x_{p+1},\dots, \bar x_n$ such that $\bar x := (\bar x_1,\dots, \bar x_n) \in \P$ and $f(\bar x_1,\dots,\bar x_p) = x^\transp H x + h^\transp x$.
% recover an optimal solution to \cref{th main opt}.
%\note{might need here only boundedness of the problem}

At this point, it seems possible to obtain an alternative proof of \cref{th main opt} that does not use \cref{prop main feasibility,prop sandwich}, by applying oracle-based techniques.
% similar to those in \cite{Dadush12,OerWagWei14,ChiGriMalParVesZol19,GriMal19,VesGriZolChi20}.
% used in \cite{Dadush12,OerWagWei14,VesGriZolChi20}.
% to prove \cref{th main opt}.
%However, it seems that a significant amount of work is still necessary, as we elaborate below.
%It does not seem to be possible to directly apply the theorems in \cite{OerWagWei14,VesGriZolChi20} to obtain \cref{th main opt} as a direct corollary.
However, we are not aware of a theorem in the literature that directly implies \cref{th main opt} without the need for some additional work.
In fact, 
%the formal statements in the literature 
the known results in the literature present at least one of the following two drawbacks:
%fall into at least one of the following two cases:
\begin{itemize}
\item[(i)] they require the construction of an extension $f'$ of $f$ to $\R^p$ or to a ball in $\R^p$, such that $f'$ is convex (theorem~1 in~\cite{OerWagWei14}), conic (theorem~14 in ~\cite{ChiGriMalParVesZol19}) or discrete convic (theorem~1 in~\cite{VesGriZolChi20});
\item[(ii)] they show that the running time is FPT only in expectation, as opposed to worst-case (theorem~7.5.1 in~\cite{Dadush12}, theorem~10 in \cite{GriMal19}), or polynomial time in fixed dimension, which is weaker than FPT (theorem~1 in~\cite{OerWagWei14}).
\end{itemize}
Furthermore, most of these results require a subgradient oracle for $f$ (theorem~7.5.1 in~\cite{Dadush12}, theorem~10 in \cite{GriMal19}) or for an extension of $f$ (theorem~1 in~\cite{OerWagWei14}).
We think it is an interesting open question to understand whether it is possible to prove a general result in integer convex optimization that directly implies \cref{th main opt}.

%(iv) separation oracle on $\proj_p (\P)$ (theorem~7.5.1 in~\cite{Dadush12})

%We justify this statement for the works \cite{Dadush12,OerWagWei14,VesGriZolChi20} which present, to the best of our knowledge, the most closely related results.
%First we consider \cite{Dadush12}.

%Among all available proofs, we choose to present ours proof because 
%it does not use any strong hammer, 
%FPT can be easily seen since we do not use results 

%A natural idea to obtain a proof to \cref{th main opt} is to reduce \cref{prob MICQP} to a pure integer convex optimization problem in $\R^p$, where one can make use of number of known complexity and algorithmic results, including .

\ifthenelse {\boolean{SIOPT}}
{
% For SIOPT begin
\bibliographystyle{siamplain}
% For SIOPT end
}
{
% For OO begin
\bibliographystyle{plain}
% For OO end
}

%\bibliography{biblio}

\begin{thebibliography}{10}

\bibitem{AliGol03}
F.~Alizadeh and D.~Goldfarb.
\newblock Second-order cone programming.
\newblock {\em Mathematical Programming, Series B}, 95:3--51, 2003.

\bibitem{BiedPHil23MPB}
Daniel Bienstock, Alberto Del~Pia, and Robert Hildebrand.
\newblock Complexity, exactness, and rationality in polynomial optimization.
\newblock {\em Mathematical Programming, Series B}, 197:661--692, 2023.

\bibitem{BowBur74}
V.J. Bowman and Claude-Alain Burdet.
\newblock On the general solution to systems of mixed-integer linear equations.
\newblock {\em SIAM Journal on Applied Mathematics}, 26(1):120--125, 1974.

\bibitem{ChiGriMalParVesZol19}
A.Yu. Chirkov, D.V. Gribanov, D.S. Malyshev, P.M. Pardalos, S.I. Veselov, and
  N.Yu. Zolotykh.
\newblock On the complexity of quasiconvex integer minimization problem.
\newblock {\em Journal of Global Optimization}, 73(4):761--788, 2019.

\bibitem{ConCorZamBook}
Michele Conforti, G\'erard Cornu\'ejols, and Giacomo Zambelli.
\newblock {\em Integer Programming}.
\newblock Springer, 2014.

\bibitem{Dadush12}
D.~Dadush.
\newblock {\em Integer Programming, Lattice Algorithms, and Deterministic
  Volume Enumeration}.
\newblock PhD thesis, Georgia Institute of Technology, 2012.

\bibitem{dP16IPCO}
Alberto Del~Pia.
\newblock On approximation algorithms for concave mixed-integer quadratic
  programming.
\newblock In {\em Proceedings of IPCO 2016}, volume 9682 of {\em Lecture Notes
  in Computer Science}, pages 1--13. Springer, 2016.

\bibitem{dP18MPB}
Alberto Del~Pia.
\newblock On approximation algorithms for concave mixed-integer quadratic
  programming.
\newblock {\em Mathematical Programming, Series B}, 172(1--2):3--16, 2018.

\bibitem{dP19SIOPT}
Alberto Del~Pia.
\newblock Subdeterminants and concave integer quadratic programming.
\newblock {\em SIAM Journal on Optimization}, 29(4):3154--3173, 2019.

\bibitem{dP23bMPA}
Alberto Del~Pia.
\newblock An approximation algorithm for indefinite mixed integer quadratic
  programming.
\newblock {\em Mathematical Programming, Series A}, 201:263--293, 2023.

\bibitem{dPDeyMol17MPA}
Alberto Del~Pia, Santanu~S. Dey, and Marco Molinaro.
\newblock Mixed-integer quadratic programming is in {NP}.
\newblock {\em Mathematical Programming, Series A}, 162(1):225--240, 2017.

\bibitem{GriMal19}
D.V. Gribanov and D.S. Malyshev.
\newblock Integer conic function minimization based on the comparison oracle.
\newblock In {\em Proceedings of the International Conference on Mathematical
  Optimization Theory and Operations Research (MOTOR 2019)}, pages 218--231.
  Springer, Cham, 2019.

\bibitem{GroLovSch88}
Martin Gr\"otschel, L\'aszl\'o Lov\'asz, and Alexander Schrijver.
\newblock {\em Geometric Algorithms and Combinatorial Optimization}.
\newblock Springer-Verlag, Berlin, 1988.

\bibitem{HoKil17}
N.~Ho-Nguyen and F.~Kılın{\c c}-Karzan.
\newblock A second-order cone based approach for solving the trust region
  subproblem and its variants.
\newblock {\em SIAM Journal on Optimization}, 27(3):1485--1512, 2017.

\bibitem{KanBac79}
Ravindran Kannan and Achim Bachem.
\newblock Polynomial algorithms for computing the {S}mith and {H}ermite normal
  forms of an integer matrix.
\newblock {\em SIAM Journal on Computing}, 8(4):499--507, 1979.

\bibitem{Kha79}
Leonid~G. Khachiyan.
\newblock A polynomial algorithm in linear programming (in {R}ussian).
\newblock {\em Doklady Akademii Nauk SSSR}, 244:1093--1096, 1979.
\newblock (English translation: \emph{Soviet Mathematics Doklady}, 20:191--194,
  1979).

\bibitem{Kha83}
Leonid~G. Khachiyan.
\newblock Convexity and complexity in polynomial programming.
\newblock In {\em Proceedings of the International Congress of Mathematicians},
  pages 16--24, Warsaw, 1983.

\bibitem{KozTarKha81}
Mikhail~K. Kozlov, Sergey~P. Tarasov, and Leonid~G. Khachiyan.
\newblock The polynomial solvability of convex quadratic programming.
\newblock {\em USSR Computational Mathematics and Mathematical Physics},
  20(5):223--228, 1981.

\bibitem{Len83}
Hendrik W.~Jr. Lenstra.
\newblock Integer programming with a fixed number of variables.
\newblock {\em Mathematics of Operations Research}, 8(4):538--548, 1983.

\bibitem{Mey73}
Jr. Meyer, Carl~D.
\newblock Generalized inverses and ranks of block matrices.
\newblock {\em SIAM Journal on Applied Mathematics}, 25(4):597--602, 1973.

\bibitem{MunSer22}
Gonzalo Mu\~noz and Felipe Serrano.
\newblock Maximal quadratic-free sets.
\newblock {\em Mathematical Programming, Series B}, 192:229--270, 2022.

\bibitem{MurKab87}
Katta~G. Murty and Santosh~N. Kabadi.
\newblock Some {NP}-complete problems in quadratic and linear programming.
\newblock {\em Mathematical Programming}, 39:117--129, 1987.

\bibitem{OerWagWei14}
T.~Oertel, C.~Wagner, and Robert Weismantel.
\newblock Integer convex minimization by mixed integer linear optimization.
\newblock {\em Operations Research Letters}, 42:424--428, 2014.

\bibitem{SchBookIP}
Alexander Schrijver.
\newblock {\em Theory of Linear and Integer Programming}.
\newblock Wiley, Chichester, 1986.

\bibitem{Vav90}
Stephen~A. Vavasis.
\newblock Quadratic programming is in {NP}.
\newblock {\em Information Processing Letters}, 36:73--77, 1990.

\bibitem{VesGriZolChi20}
S.I. Veselov, D.V. Gribanov, N.Yu. Zolotykh, and A.Yu. Chirkov.
\newblock A polynomial algorithm for minimizing discrete convic functions in
  fixed dimension.
\newblock {\em Discrete Applied Mathematics}, 2020.

\end{thebibliography}

\end{document}